\documentclass[sn-mathphys,Numbered]{sn-jnl}

\usepackage{graphicx}%
\usepackage{multirow}%
\usepackage{amsmath,amssymb,amsfonts}%
\usepackage{amsthm}%
\usepackage{mathrsfs}%
\usepackage[title]{appendix}%
\usepackage{xcolor}%
\usepackage{textcomp}%
\usepackage{manyfoot}%
\usepackage{booktabs}%
\usepackage{algorithm}%
\usepackage{algorithmicx}%
\usepackage{algpseudocode}%
\usepackage{listings}%

\theoremstyle{thmstyleone}%
\newtheorem{theorem}{Theorem}
\theoremstyle{thmstyletwo}%
\newtheorem{example}{Example}%
\newtheorem{remark}{Remark}%
\newtheorem{lemma}{Lemma}%
\newtheorem{corollary}{Corollary}%

\theoremstyle{thmstylethree}%

\begin{document}

\title[On blow-up conditions]{On blow-up conditions for nonlinear higher order evolution inequalities}


\author*[1,2]{\fnm{A. A.} \sur{Kon'kov}}\email{konkov@mech.math.msu.su}

\author[2]{\fnm{A. E.} \sur{Shishkov}}\email{aeshkv@yahoo.com}
\equalcont{These authors contributed equally to this work.}

\affil*[1]{\orgdiv{Department of Differential Equations}, \orgname{Faculty of Mechanics and Mathematics, Lo\-mo\-no\-sov Mo\-s\-cow State University}, \orgaddress{\street{Vorobyovy Gory}, \city{Moscow}, \postcode{119991}, \country{Russia}}}

\affil[2]{\orgdiv{Center of Nonlinear Problems of Mathematical Physics}, \orgname{RUDN University}, \orgaddress{\street{Miklukho-Maklaya str. 6}, \city{Moscow}, \postcode{117198}, \country{Russia}}}

\abstract{
For the problem
$$
	\left\{
		\begin{aligned}
			&
			\partial_t^k u
			-
			\sum_{|\alpha| = m}
			\partial^\alpha
			a_\alpha (x, t, u)
			\ge
			f (|u|)
			\quad
			\mbox{in } 
			{\mathbb R}_+^{n+1} = {\mathbb R}^n \times (0, \infty),
			\\
			&
			u (x, 0) = u_0 (x),
			\:
			\partial_t u (x, 0) = u_1 (x),
			\ldots,
			\partial_t^{k-1} u (x, 0) = u_{k-1} (x) \ge 0,
		\end{aligned}
	\right.
$$
where $u_i \in L_{1, loc} ({\mathbb R}^n)$ and $a_\alpha$ are Caratheodory functions such that
$$
	|a_\alpha (x, t, \zeta)| \le A |\zeta|^p,
	\quad
	A, p = const > 0,
$$
for almost all $(x, t) \in {\mathbb R}_+^{n+1}$ and for all $\zeta \in {\mathbb R}$, we obtain exact conditions on the function $f$ guaranteeing that any global weak solution is identically zero.
}




\maketitle

\section{Introduction}\label{sec1}

We study solutions of the inequality
\begin{equation}
		\partial_t^k u
		-
		{\mathcal L} u
		\ge
		f (|u|)
		\quad
		\mbox{in } 
		{\mathbb R}_+^{n+1} = {\mathbb R}^n \times (0, \infty)
	\label{1.1i}
\end{equation}
satisfying the initial conditions
\begin{equation}
	u (x, 0) = u_0 (x),
	\partial_t u (x, 0) = u_1 (x),
	\ldots,
	\partial_t^{k-1} u (x, 0) = u_{k-1} (x) \ge 0,
	\label{1.1c}
\end{equation}
where $k \ge 1$ and $n \ge 1$ are integers and 
$$
	{\mathcal L} u
	=
	\sum_{|\alpha| = m}
	\partial^\alpha
	a_\alpha (x, t, u)
$$
is a differential operator of order $m \ge 1$ whose coefficients are Caratheodory functions such that
$$
	|a_\alpha (x, t, \zeta)| \le A |\zeta|^p,
	\quad
	|\alpha| = m,
$$
with some constants $A > 0$ and $p > 0$ for almost all 
$(x, t) \in {\mathbb R}_+^{n+1}$ and for all $\zeta \in {\mathbb R}$.
As is customary, by $\alpha$ we mean a multi-index $\alpha = {(\alpha_1, \ldots, \alpha_n)}$. 
In so doing,
$|\alpha| = \alpha_1 + \ldots + \alpha_n$ 
and
$
	\partial^\alpha 
	= 
	{\partial^{|\alpha|} / (\partial_{x_1}^{\alpha_1} \ldots \partial_{x_n}^{\alpha_n})}.
$

Let us denote by $B_r$ the open ball in ${\mathbb R}^n$ of radius $r > 0$ and center at zero. Also let $Q_r^\varkappa = B_r \times (0, r^\varkappa)$, $\varkappa > 0$.
We say that a function $u$ belongs to $L_{\lambda, loc} ({\mathbb R}^n)$, $0 < \lambda \le \infty$, if $u \in L_\lambda (B_r)$ for all real numbers $r > 0$. 
In its turn, $u \in L_{\lambda, loc} ({\mathbb R}^{n+1}_+)$ if $u \in L_\lambda (Q_r^1)$ for all real numbers $r > 0$. 

We assume that $u_i \in L_{1, loc} ({\mathbb R}^n)$, $i = 0, \ldots, k - 1$. Also let  $f (\zeta)$ and $f (\zeta^{1/p})$ are non-decreasing convex functions on the interval $[0, \infty)$ such that $f (\zeta) > 0$ for all $\zeta \in (0, \infty)$.

A function 
$
	u 
	\in 
	{L_{1, loc} ({\mathbb R}_+^{n+1})}
	\cap
	{L_{p, loc} ({\mathbb R}_+^{n+1})}
$
is called a global weak solution of problem~\eqref{1.1i}, \eqref{1.1c}
if ${f (|u|)} \in {L_{1, loc} ({\mathbb R}^{n+1}_+)}$ and
\begin{align}
	&
	\int_{
		{\mathbb R}_+^{n+1}
	}
	(-1)^k
	u
	\partial_t^k \varphi
	dx
	dt
	-
	\int_{
		{\mathbb R}_+^{n+1}
	}
	\sum_{|\alpha| = m}
	(- 1)^m
	a_\alpha (x, t, u)
	\partial^\alpha
	\varphi
	dx
	dt
	\nonumber
	\\
	&
	\qquad
	\ge
	\int_{
		{\mathbb R}^n
	}
	\sum_{i=0}^{k-1}
	(-1)^i
	u_{k - 1 - i} (x)
	\partial_t^i \varphi (x, 0)
	dx
	+
	\int_{
		{\mathbb R}_+^{n+1}
	}
	f (|u|)
	\varphi
	dx
	dt
	\label{1.2}
\end{align}
for any non-negative function $\varphi \in C_0^\infty ({\mathbb R}^{n+1})$.
Analogously, a function 
$
	u 
	\in 
	{L_{1, loc} ({\mathbb R}_+^{n+1})}
	\cap
	{L_{p, loc} ({\mathbb R}_+^{n+1})}
$
is a global weak solution of the equation
\begin{equation}
	\partial_t^k u
	-
	{\mathcal L} u
	=
	f (|u|)
	\quad
	\mbox{in } 
	{\mathbb R}_+^{n+1}
	\label{1.3e}
\end{equation}
satisfying conditions~\eqref{1.1c} if ${f (|u|)} \in {L_{1, loc} ({\mathbb R}^{n+1}_+)}$ and
\begin{align*}
	&
	\int_{
		{\mathbb R}_+^{n+1}
	}
	(-1)^k
	u
	\partial_t^k \varphi
	dx
	dt
	-
	\int_{
		{\mathbb R}_+^{n+1}
	}
	\sum_{|\alpha| = m}
	(- 1)^m
	a_\alpha (x, t, u)
	\partial^\alpha
	\varphi
	dx
	dt
	\nonumber
	\\
	&
	\qquad
	=
	\int_{
		{\mathbb R}^n
	}
	\sum_{i=0}^{k-1}
	(-1)^i
	u_{k - 1 - i} (x)
	\partial_t^i \varphi (x, 0)
	dx
	+
	\int_{
		{\mathbb R}_+^{n+1}
	}
	f (|u|)
	\varphi
	dx
	dt
\end{align*}
for any $\varphi \in C_0^\infty ({\mathbb R}^{n+1})$.
It is obvious that every global weak solution of~\eqref{1.3e}, \eqref{1.1c} is also a global weak solution of~\eqref{1.1i}, \eqref{1.1c}.

As an example of~\eqref{1.3e}, we can take the nonlinear diffusion equation
$$
	u_t = \Delta u^p + f (u)
	\quad
	\mbox{in } 
	{\mathbb R}_+^{n+1}
$$
which can be represented in the form
\begin{equation}
	u_t 
	=
	\operatorname{div}
	\left(
		D
		\nabla u
	\right)
	+
	f (u)
	\quad
	\mbox{in } 
	{\mathbb R}_+^{n+1},
	\label{1.6}
\end{equation}
where $D = p u^{p - 1}$ is a diffusion coefficient depending on the density $u > 0$ 
in a power-law manner and $f (u)$ is a source-density function.
In the case of $p \ge 1$, it is customary to say that~\eqref{1.6} is the slow diffusion equation. In its turn, if $0 < p < 1$, then one says that~\eqref{1.6} is the fast diffusion equation.

The absence of non-trivial solutions of differential equations and inequalities or, in other words, the blow-up phenomenon was studied by many authors~\cite{CPX, CG, EGKPCR, EGKPMN, F, GS, H, Kato, Keller, KK, KV, mePRSE, Nonlinearity, meIzv, LS, MPbook, Osserman, PaTa, PoTe, PoVe, PS, SGKM, SV}.
In most cases, these studies were limited to power-law nonlinearities in the right-hand side of inequality~\eqref{1.1i} or dealt with second-order differential operators.
In our paper, we deal with 
a much wider class 
of differential inequalities.
Our aim is to obtain exact conditions guaranteeing that any global weak solution of~\eqref{1.1i}, \eqref{1.1c} is trivial, i.e. is equal to zero almost everywhere in ${\mathbb R}_+^{n+1}$.

In paper~\cite{LS}, in the partial case where $k = 1$ and ${\mathcal L} = \Delta$, it was obtained a  blow-up condition of a Dini-type structure that strengthens the well-known result of Fujita--Hayakawa. Moreover, it was shown in~\cite{LS} that the validity of this condition is necessary for the non-existence of non-trivial global solutions.
Our general condition~\eqref{T2.1.1} coincides with that in~\cite{LS} in the mentioned partial case.
Let us note that, in contrast to~\cite{LS}, we impose no ellipticity conditions on the differential operator ${\mathcal L}$ and no growth conditions on the initial data of the Cauchy problem and on a solution of this problem itself.
In so doing, we have reason to believe that~\eqref{T2.1.1} is exact in much more general case.
For example, in the case of the Cauchy problem for the wave equation, i.e. in the case where $k = 2$ and ${\mathcal L} = \Delta$, condition~\eqref{T2.1.1} is more general than the well-known Kato blow-up condition~\cite{Kato}.
In accordance with~\cite[Theorem~2]{PoVe} the Kato blow-up condition is exact in the scale of power-law nonlinearities for weak solutions of the Cauchy problem for the wave equation.

For $f (t) = t^\lambda$, i.e. in the case of power-law nonlinearities, Theorem~\ref{T2.1} proved in our paper implies the result obtained in~\cite[Theorem~1]{EGKPCR} (see Example~\ref{E2.1}). 
In this case, conditions~\eqref{1.4} and~\eqref{1.5} take the simple form $\lambda > 1$ and $\lambda > p$, respectively.

We also note that~\eqref{1.5} coincides with the analogous condition in \cite[Theorem~2.1]{Nonlinearity}, where the blow-up phenomenon was studied for 
the stationary inequality
\begin{equation}
	- {\mathcal L} u
	\ge
	f (|u|)
	\quad
	\mbox{in } 
	{\mathbb R}^n.
	\label{1.7}
\end{equation}
This seems natural because the existence of a non-trivial solution of the last inequality implies the existence of a non-trivial solution of problem~\eqref{1.1i}, \eqref{1.1c} that does not depend on $t$. 
However, condition~\eqref{T2.1.1} is stronger than the similar blow-up condition for solutions of inequality~\eqref{1.7}. This effect can be detected even in the case of power-law nonlinearities.
For instance, if $f (t) = t^\lambda$ and, at the same time, $p \ge 1$ and $n > m$, then in accordance with~\cite[Theorems 4.1 and~29.1]{MPbook} the blow-up of solutions of inequality~\eqref{1.7} is guaranteed by the condition
\begin{equation}
	p
	<
	\lambda
	\le
	\frac{n}{n - m} p,
	\label{1.8}
\end{equation}
whereas the blow-up of solutions of the problem
\begin{equation}
	u_t
	-
	{\mathcal L} u
	\ge
	|u|^\lambda
	\quad
	\mbox{in } 
	{\mathbb R}_+^{n+1},
	\quad
	u (x, 0) = u_0 (x) \ge 0,
	\label{E2.1.1}
\end{equation}
occurs if
\begin{equation}
	p 
	<
	\lambda
	\le
	p
	+
	\frac{m}{n}.
	\label{1.9}
\end{equation}
In so doing, both conditions~\eqref{1.8} and~\eqref{1.9} are exact. Moreover, if $p = 1$ and $m = 2$, then~\eqref{1.9} coincides with the well-known Fujita--Hayakawa condition.
As mentioned above, the first inequality in~\eqref{1.9} follows from~\eqref{1.5}, while the second inequality is a consequence of~\eqref{T2.1.1}.

Now, let us say a few words about the structure of this paper. Our main results are Theorems~\ref{T2.0}--\ref{T2.3} formulated in Section~\ref{main}. This section also contains some examples and remarks demonstrating the exactness of these results. 

Sections~\ref{preliminaries} and~\ref{proof} are devoted to the proof of Theorems~\ref{T2.0}--\ref{T2.3}. The main idea of the proof is to estimate the function $E$ defined by formula~\eqref{E}.
The final aim is to show that the integrand in the right-hand sides of~\eqref{E} is equal to zero. This is obviously equivalent to the triviality of a solution of~\eqref{1.1i}, \eqref{1.1c}. 

The key statements of Section~\ref{preliminaries} are Lemmas~\ref{L3.4}--\ref{L3.7}. They imply Lemmas~\ref{L3.8} and~\ref{L3.9} that contain estimates of the function $E$.
Lemmas~\ref{L3.1}--\ref{L3.3} are of a technical nature. We need them to prove 
Lemmas~\ref{L3.4}--\ref{L3.7}.

Theorems~\ref{T2.0}--\ref{T2.3} are proved in Section~\ref{proof}.
According to Lemmas~\ref{L3.4} and~\ref{L3.5}, in the case of $n \le m (1 - 1 / k)$, conditions~\eqref{1.4} and~\eqref{1.5} guarantee that $E (r)$ tends to zero as $r \to \infty$. This is obviously enough to prove Theorem~\ref{T2.0}. 
To prove Theorems~\ref{T2.1} and~\ref{T2.2}, we assume by contradiction that~\eqref{1.1i}, \eqref{1.1c} has a non-trivial solution.
This, in turn, allows one to assert that 
$E (r) > 0$ for all sufficiently large $r > 0$. 
Applying 
Lemmas~\ref{L3.4}, \ref{L3.5}, \ref{L3.8}, and~\ref{L3.9}, 
we obtain estimates of 
$
	E (r)
$
divided by a suitable power of $r$ in a neighborhood of zero.
These estimates lead us to a contradiction with conditions of Theorems~\ref{T2.1} and~\ref{T2.2}.
Finally, we note that Theorem~\ref{T2.3} is a consequence of Theorems~\ref{T2.0}--\ref{T2.2}.

\section{Main results}\label{main}

\begin{theorem}\label{T2.0}
Let $n \le m (1 - 1 / k)$ and, moreover,
\begin{equation}
	\int_1^\infty
	\frac{
		\zeta^{1 / k - 1}
		d\zeta
	}{
		f^{1 / k} (\zeta)
	}
	<
	\infty
	\label{1.4}
\end{equation}
and
\begin{equation}
	\int_1^\infty
	\frac{
		\zeta^{p / m - 1}
		d\zeta
	}{
		f^{1 / m} (\zeta)
	}
	<
	\infty.
	\label{1.5}
\end{equation}
Then any global weak solution of~\eqref{1.1i}, \eqref{1.1c} is trivial.
\end{theorem}

\begin{theorem}\label{T2.1}
Let $n > m (1 - 1 / k)$ and $p \in (0, \infty) \cap (1 - m / n, \infty)$. If~\eqref{1.4} and~\eqref{1.5} are valid and
\begin{equation}
	\int_0^1
	\frac{
		f (r)
		dr
	}{
		r^{1 + \gamma}
	}
	=
	\infty,
	\label{T2.1.1}
\end{equation}
where
\begin{equation}
	\gamma
	=
	\frac{
		n p + m / k
	}{
		n - (1 - 1 / k) m
	},
	\label{T2.1.2}
\end{equation}
then any global weak solution of~\eqref{1.1i}, \eqref{1.1c} is trivial.
\end{theorem}

\begin{theorem}\label{T2.2}
Let $n > m (1 - 1 / k)$ and $p \in (0,  1 - m / n]$. If~\eqref{1.4} and~\eqref{1.5} are valid and
\begin{equation}
	\lim_{r \to +0}
	\frac{
		f (r)
	}{
		r^\mu
	}
	=
	\infty
	\label{T2.2.1}
\end{equation}
for some real number $\mu < p / (1 - m /n)$,
then any global weak solution of~\eqref{1.1i}, \eqref{1.1c} is trivial.
\end{theorem}

\begin{theorem}\label{T2.3}
Let ~\eqref{1.4} and~\eqref{1.5} be valid and $f(0) > 0$. Then~\eqref{1.1i}, \eqref{1.1c} has no global weak solutions. 
\end{theorem}

Proof of Theorems~\ref{T2.0}--\ref{T2.3} is given in Section~\ref{proof}.

\begin{remark}\label{R2.1}
Inequality~\eqref{1.4} coincides with the Kiguradze--Kvinikadze blow-up condition for solutions of the Cauchy problem
\begin{equation}
	w^{(k)} (r) = f (w),
	\quad
	w (0) > 0,
	\:
	w^{(i)} (0) = 0,
	\:
	i = 1, \ldots, k - 1,
	\label{R2.1.1}
\end{equation}
for the ordinary differential equation~\cite[Theorem~1.2]{KK}.
In the case where~\eqref{1.4} is not fulfilled, it can be shown that~\eqref{R2.1.1} has a solution defined on the whole interval $[0, \infty)$.
Therefore, taking $u (x, t) = w (t)$, we obtain a non-trivial global weak solution of~\eqref{1.1i}, \eqref{1.1c} which does not depend on spatial variables. 

Analogously,~\eqref{1.5} coincides with the Kiguradze--Kvinikadze blow-up condition for the Cauchy problem
\begin{equation}
	w^{(m)} (r) = f (w^{1 / p}),
	\quad
	w (0) > 0,
	\:
	w^{(i)} (0) = 0,
	\:
	i = 1, \ldots, m - 1.
	\label{R2.1.2}
\end{equation}
If~\eqref{1.5} is not valid, then~\eqref{R2.1.2} has a solution defined on the whole interval $[0, \infty)$. 
Since $w (r) > w (0)$ for all $r \in (0, \infty)$,
setting $v = w - w (0)$, we obtain a solution of the inequality
$$
	v^{(m)} (r) \ge f (v^{1 / p})
$$
for all $r \in [0, \infty)$ satisfying the conditions
$v^{(i)} (0) = 0$, $i = 0, \ldots, m - 1$.
Thus, the function 
$u (x, t) = v^{1 / p} (|x_1|)$ 
is a non-trivial global weak solution of~\eqref{1.1i}, \eqref{1.1c} with
$$
	{\mathcal L} u
	=
	-
	\frac{
		\partial^m |u|^p 
		(\operatorname{sign} x_1)^m
	}{
		\partial x_1^m
	}.
$$

In the case where $k = p = 1$ and $m = 2$, we obviously have $\gamma = 1 + 2 / n$. In this case,~\eqref{T2.1.1} coincides with the blow-up condition for solutions of the Cauchy problem for the heat equation obtained in~\cite{LS}.
Note that, in paper~\cite{LS}, it is assumed that the initial values of the Cauchy problem belongs to $L_\infty ({\mathbb R}^n)$ and the solutions belong to ${L_\infty ({\mathbb R}^n \times (0, T))}$ for all real number $T > 0$.
We do not need such assumptions.
\end{remark}

For the problem 
\begin{equation}
			u_{tt} - \Delta u \ge f (|u|)
			\quad
			\mbox{in } {\mathbb R}_+^{n+1},
			\quad
			u (x, 0) = u_0 (x),
			\:
			u_t (x, 0) = u_1 (x) \ge 0.
	\label{WE}
\end{equation}
Theorems~\ref{T2.0} and~\ref{T2.1} imply the following statements.

\begin{corollary}\label{C2.2}
Let $n = 1$ and
\begin{equation}
	\int_1^\infty
	(f (\zeta) \zeta)^{-1/2}
	d\zeta
	<
	\infty.
	\label{C2.2.1}
\end{equation}
Then any global weak solution of~\eqref{WE} is trivial.
\end{corollary}

\begin{corollary}\label{C2.3}
Let $n \ge 2$, condition~\eqref{C2.2.1} be valid, and
\begin{equation}
	\int_0^1
	\frac{
		f(r)
		dr
	}{
		r^{2 + 2 /(n - 1)}
	}
	=
	\infty.
	\label{C2.3.1}
\end{equation}
Then any global weak solution of~\eqref{WE} is trivial.
\end{corollary}

To prove Corollaries~\ref{C2.2} and~\ref{C2.3}, we note that, for $p = 1$ and $k =  m = 2$, condition~\eqref{C2.2.1} is equivalent to~\eqref{1.4} and~\eqref{1.5}.
At the same time,~\eqref{C2.3.1} is equivalent to~\eqref{T2.1.1}.
Thus, Corollary~\ref{C2.2} follows from Theorem~\ref{T2.0}, while Corollary~\ref{C2.3} is a consequence of Theorem~\ref{T2.1}.

\begin{remark}\label{R2.2}
As mentioned above, condition~\eqref{C2.2.1} is essential.
If~\eqref{C2.2.1} is not valid, then problem~\eqref{WE} has a positive global weak solution which depends only on $t$.
For power-law nonlinearities $f (t) = t^\lambda$, Corollary~\ref{C2.3} leads to the Kato blow-up condition
\begin{equation}
	1
	<
	\lambda
	\le
	1 
	+
	\frac{2}{n - 1}
	\label{R2.2.1}
\end{equation}
obtained in paper~\cite{Kato}. 
Indeed, the first inequality in~\eqref{R2.2.1} is equivalent to~\eqref{C2.2.1}, whereas the second inequality is equivalent to~\eqref{C2.3.1}.
It is well-known that both inequalities in~\eqref{R2.2.1} are exact.
To verify the exactness of the first inequality, it is sufficient to check that, in the case of $\lambda \le 1$, the function
$
	u (x, t)
	=
	e^t
$
is a positive global weak solution of the problem
$$
	u_{tt} - \Delta u \ge u^\lambda
	\quad
	\mbox{in } {\mathbb R}_+^{n+1},
	\quad
	u (x, 0) = u_0 (x),
	\:
	u_t (x, 0) = u_1 (x) \ge 0.
$$
In accordance with~\cite[Theorem~2]{PoVe} the second inequality in~\eqref{R2.2.1} is also exact.
\end{remark}

\begin{example}\label{E2.1}
Let us consider problem~\eqref{E2.1.1}.
It is easy to see that condition~\eqref{1.4} is equivalent to $\lambda > 1$, while~\eqref{1.5} is equivalent to $\lambda > p$.
In so doing,~\eqref{T2.1.1} is fulfilled if 
$$
	\lambda \le p + \frac{m}{n}.
$$ 
Thus, in accordance with Theorem~\ref{T2.1} if
\begin{equation}
	\max \{ 1, p \}
	<
	\lambda 
	\le 
	p + \frac{m}{n},
	\label{E2.1.2}
\end{equation}
then any global weak solution of~\eqref{E2.1.1} is trivial. 

Note that condition~\eqref{E2.1.2} was earlier obtained in~\cite[Theorem~1]{EGKPCR}.
\end{example}

\begin{example}\label{E2.2}
We examine the critical exponent $\lambda = p + m / n > 1$ in~\eqref{E2.1.2}.
Namely, consider the problem
\begin{equation}
	u_t
	-
	{\mathcal L} u
	\ge
	u^{p + m / n}
	\log^\mu 
	\left(
		e + \frac{1}{u}
	\right)
	\quad
	\mbox{in } 
	{\mathbb R}_+^{n+1},
	\quad
	u (x, 0) = u_0 (x) \ge 0,
	\label{E2.2.1}
\end{equation}
where $\mu$ is a real number. 
For $u = 0$, we assume by continuity that the right-hand side of inequality~\eqref{E2.2.1} is equal to zero.
By Theorem~\ref{T2.1}, if
$$
	\mu \ge -1,
$$
then any global weak solution of~\eqref{E2.2.1} is trivial. 
\end{example}

\section{Preliminary and auxiliary statements}\label{preliminaries}

In this section, it is assumed that $u$ is a global weak solution of~\eqref{1.1i}, \eqref{1.1c} and, moreover, conditions~\eqref{1.4} and~\eqref{1.5} are valid.
Let us recall that we have agreed to denote by $Q_r^\varkappa$ the cylinder $Q_r^\varkappa = B_r \times (0, r^\varkappa)$, where $\varkappa > 0$ is a real number and $B_r$ is the open ball of radius $r > 0$ centered at zero.
Below, by $C$ and $\sigma$ we mean various positive constants that can depend only on $A$, $k$, $m$, $n$, $p$, and $\varkappa$. 
Put
$$
	F_\varkappa (s)
	=
	\left(
		\int_s^\infty
		\frac{
			\zeta^{1 / k - 1}
			d\zeta
		}{
			f^{1 / k} (\zeta)
		}
	\right)^{1 / \varkappa}
	+
	\int_s^\infty
	\frac{
		\zeta^{p / m - 1}
		d \zeta
	}{
		f^{1 / m} (\zeta)
	}.
$$

\begin{lemma}\label{L3.1}
Let $R > 0$ and $\varkappa > 0$ be real numbers. Then
\begin{align}
	\frac{
		1
	}{
		(r_2^\varkappa - r_1^\varkappa)^k
	}
	\int_{
		r_1^\varkappa
	}^{
		r_2^\varkappa
	}
	\int_{B_{r_2}}
	|u|
	dx
	dt
	&
	{}
	+
	\frac{1}{(r_2 - r_1)^m}
	\int_0^{r_2^\varkappa}
	\int_{B_{r_2} \setminus B_{r_1}}
	|u|^p
	dx
	dt
	\nonumber
	\\
	&
	{}
	\ge
	C
	\int_{Q_{r_1}^\varkappa}
	f (|u|)
	dx
	dt
	\label{L3.1.1}
\end{align}
for all real numbers $R \le r_1 < r_2 \le 2 R$.
\end{lemma}

\begin{proof}
Take a non-decreasing function $\varphi_0 \in C^\infty ({\mathbb R})$ such that
$$
	\left.
		\varphi_0
	\right|_{
		(- \infty, 0]
	}
	=
	0
	\quad
	\mbox{and}
	\quad
	\left.
		\varphi_0
	\right|_{
		[1, \infty)
	}
	=
	1.
$$
We put
\begin{equation}
	\varphi (x, t)
	=
	\varphi_0 
	\left(
		\frac{r_2 - |x|}{r_2 - r_1}
	\right)
	\varphi_0 
	\left(
		\frac{
			r_2^\varkappa - t
		}{
			r_2^\varkappa - r_1^\varkappa
		}
	\right).
    \label{PL3.1.1}
\end{equation}
It is easy to see that
$$
	\left|
		\int_{
			{\mathbb R}_+^{n+1}
		}
		u
		\partial_t^k \varphi
		dx
		dt
	\right|
	\le
	\frac{
		C
	}{
		(r_2^\varkappa - r_1^\varkappa)^k
	}
	\int_{
		r_1^\varkappa
	}^{
		r_2^\varkappa
	}
	\int_{B_{r_2}}
	|u|
	dx
	dt,
$$
$$
	\left|
		\int_{
			{\mathbb R}_+^{n+1}
		}
		\sum_{|\alpha| = m}
		(- 1)^m
		a_\alpha (x, t, u)
		\partial^\alpha
		\varphi
		dx
		dt
	\right|
	\le
	\frac{C}{(r_2 - r_1)^m}
	\int_0^{r_2^\varkappa}
	\int_{B_{r_2} \setminus B_{r_1}}
	|u|^p
	dx
	dt,
$$
and
$$
	\int_{
		{\mathbb R}_+^{n+1}
	}
	f (|u|)
	\varphi
	dx
	dt
	\ge
	\int_0^{r_1^\varkappa}
	\int_{B_{r_1}}
	f (|u|)
	dx
	dt
	=
	\int_{Q_{r_1}^\varkappa}
	f (|u|)
	dx
	dt.
$$
Thus, taking~\eqref{PL3.1.1} as a test function in~\eqref{1.2}, we readily obtain~\eqref{L3.1.1}.
\end{proof}

\begin{lemma}\label{L3.2}
Let $p \ge 1$, then for all real numbers $R > 0$, $\varkappa > 0$, and $R \le r_1 < r_2 \le 2 R$ at least one of the following two inequalities is valid:
\begin{equation}
	I (r_2) - I (r_1)
	\ge
	C
	(r_2 - r_1)^{k p}
	R^{
		k p (\varkappa - 1)
	}
	f^p (I^{1 / p} (r_1)),
	\label{L3.2.1}
\end{equation}
\begin{equation}
	I (r_2) - I (r_1)
	\ge
	C
	(r_2 - r_1)^m
	f (I^{1 / p} (r_1)),
	\label{L3.2.2}
\end{equation}
where
\begin{equation}
	I (r)
	=
	\frac{
		1
	}{
		\operatorname{mes} Q_{2 R}^\varkappa
	}
	\int_{Q_r^\varkappa}
	|u|^p
	dx
	dt,
	\quad
	R < r < 2 R.
	\label{L3.2.3}
\end{equation}
\end{lemma}

\begin{proof}
At first, assume that
\begin{equation}
	\frac{
		1
	}{
		(r_2^\varkappa - r_1^\varkappa)^k
	}
	\int_{
		r_1^\varkappa
	}^{
		r_2^\varkappa
	}
	\int_{B_{r_2}}
	|u|
	dx
	dt
	\ge
	\frac{1}{(r_2 - r_1)^m}
	\int_0^{r_2^\varkappa}
	\int_{B_{r_2} \setminus B_{r_1}}
	|u|^p
	dx
	dt.
	\label{PL3.2.1}
\end{equation}
Then Lemma~\ref{L3.1} implies the estimate
$$
	\frac{
		1
	}{
		(r_2^\varkappa - r_1^\varkappa)^k
	}
	\int_{
		r_1^\varkappa
	}^{
		r_2^\varkappa
	}
	\int_{B_{r_2}}
	|u|
	dx
	dt
	\ge
	C
	\int_{Q_{r_1}^\varkappa}
	f (|u|)
	dx
	dt.
$$
Since
$$
	\int_{
		Q_{r_2}^\varkappa \setminus Q_{r_1}^\varkappa
	}
	|u|
	dx
	dt
	\ge
	\int_{
		r_1^\varkappa
	}^{
		r_2^\varkappa
	}
	\int_{B_{r_2}}
	|u|
	dx
	dt,
$$
this yields
\begin{equation}
	\int_{
		Q_{r_2}^\varkappa \setminus Q_{r_1}^\varkappa
	}
	|u|
	dx
	dt
	\ge
	C
	(r_2^\varkappa - r_1^\varkappa)^k
	\int_{Q_{r_1}^\varkappa}
	f (|u|)
	dx
	dt.
	\label{PL3.2.5}
\end{equation}
Evaluating the left-hand side of the last expression by H\"older's inequality
$$
	(\operatorname{mes} Q_{r_2}^\varkappa \setminus Q_{r_1}^\varkappa)^{(p - 1) / p}
	\left(
		\int_{
			Q_{r_2}^\varkappa \setminus Q_{r_1}^\varkappa
		}
		|u|^p
		dx
		dt
	\right)^{1 / p}
	\ge
	\int_{
		Q_{r_2}^\varkappa \setminus Q_{r_1}^\varkappa
	}
	|u|
	dx
	dt,
$$
we obtain
$$
	\left(
		\int_{
			Q_{r_2}^\varkappa \setminus Q_{r_1}^\varkappa
		}
		|u|^p
		dx
		dt
	\right)^{1 / p}
	\ge
	\frac{
		C (r_2^\varkappa - r_1^\varkappa)^k
	}{
		(\operatorname{mes} Q_{r_2}^\varkappa \setminus Q_{r_1}^\varkappa)^{(p - 1) / p}
	}
	\int_{Q_{r_1}^\varkappa}
	f (|u|)
	dx
	dt,
$$
whence it follows that
\begin{align}
	\frac{
		1
	}{
		\operatorname{mes} Q_{2 R}^\varkappa
	}
	\int_{
		Q_{r_2}^\varkappa \setminus Q_{r_1}^\varkappa
	}
	|u|^p
	dx
	dt
	\ge
	{}
	&
	C (r_2^\varkappa - r_1^\varkappa)^{k p}
	\left(
		\frac{
			\operatorname{mes} Q_{2 R}^\varkappa
		}{
			\operatorname{mes} Q_{r_2}^\varkappa \setminus Q_{r_1}^\varkappa
		}
	\right)^{p - 1}
	\nonumber
	\\
	&
	{}
	\times
	\left(
		\frac{1}{\operatorname{mes} Q_{2 R}^\varkappa}
		\int_{Q_{r_1}^\varkappa}
		f (|u|)
		dx
		dt
	\right)^p.
	\label{PL3.2.2}
\end{align}
Since $f (\zeta^{1/p})$ is a non-decreasing convex function, we have
\begin{align}
	\frac{1}{\operatorname{mes} Q_{2 R}^\varkappa}
	\int_{Q_{r_1}^\varkappa}
	f (|u|)
	dx
	dt
	&
	\ge
	\frac{
		1
	}{
		2^{n + \varkappa}
		\operatorname{mes} Q_{r_1}^\varkappa
	}
	\int_{Q_{r_1}^\varkappa}
	f (|u|)
	dx
	dt
	\nonumber
	\\
	&
	\ge
	\frac{
		1
	}{
		2^{n + \varkappa}
	}
	f 
	\left(
		\left(
			\frac{
				1
			}{
				\operatorname{mes} Q_{r_1}^\varkappa
			}
			\int_{Q_{r_1}^\varkappa}
			|u|^p
			dx
			dt
		\right)^{1 / p}
	\right)
	\nonumber
	\\
	&
	\ge
	\frac{
		1
	}{
		2^{n + \varkappa}
	}
	f 
	\left(
		\left(
			\frac{
				1
			}{
				\operatorname{mes} Q_{2 R}^\varkappa
			}
			\int_{Q_{r_1}^\varkappa}
			|u|^p
			dx
			dt
		\right)^{1 / p}
	\right).
	\label{PL3.2.3}
\end{align}
In view of~\eqref{PL3.2.2}, this implies the estimate
\begin{align*}
	\frac{
		1
	}{
		\operatorname{mes} Q_{2 R}^\varkappa
	}
	\int_{
		Q_{r_2}^\varkappa \setminus Q_{r_1}^\varkappa
	}
	|u|^p
	dx
	dt
	\ge
	{}
	&
	C (r_2^\varkappa - r_1^\varkappa)^{k p}
	\left(
		\frac{
			\operatorname{mes} Q_{2 R}^\varkappa
		}{
			\operatorname{mes} Q_{r_2}^\varkappa \setminus Q_{r_1}^\varkappa
		}
	\right)^{p - 1}
	\\
	&
	{}
	\times
	f^p 
	\left(
		\left(
			\frac{
				1
			}{
				\operatorname{mes} Q_{2 R}^\varkappa
			}
			\int_{Q_{r_1}^\varkappa}
			|u|^p
			dx
			dt
		\right)^{1 / p}
	\right),
\end{align*}
combining which with the evident inequalities
$$
	\left(
		\frac{
			\operatorname{mes} Q_{2 R}^\varkappa
		}{
			\operatorname{mes} Q_{r_2}^\varkappa \setminus Q_{r_1}^\varkappa
		}
	\right)^{p - 1}
	\ge
	1
$$
and
\begin{equation}
	r_2^\varkappa - r_1^\varkappa
	\ge 
	C
	(r_2 - r_1) 
	R^{\varkappa - 1},
	\label{PL3.2.4}
\end{equation}
we arrive at \eqref{L3.2.1}.

Now, let~\eqref{PL3.2.1} not hold. In this case, Lemma~\ref{L3.1} implies that
\begin{equation}
	\frac{1}{(r_2 - r_1)^m}
	\int_0^{r_2^\varkappa}
	\int_{B_{r_2} \setminus B_{r_1}}
	|u|^p
	dx
	dt
	\ge
	C
	\int_{Q_{r_1}^\varkappa}
	f (|u|)
	dx
	dt,
	\label{PL3.2.6}
\end{equation}
whence in accordance with~\eqref{PL3.2.3} and the inequality
\begin{equation}
	\int_{Q_{r_2}^\varkappa \setminus Q_{r_1}^\varkappa}
	|u|^p
	dx
	dt
	\ge
	\int_0^{r_2^\varkappa}
	\int_{B_{r_2} \setminus B_{r_1}}
	|u|^p
	dx
	dt
	\label{PL3.2.7}
\end{equation}
we obtain
$$
	\frac{
		1
	}{
		\operatorname{mes} Q_{2 R}^\varkappa
	}
	\int_{Q_{r_2}^\varkappa \setminus Q_{r_1}^\varkappa}
	|u|^p
	dx
	dt
	\ge
	C
	(r_2 - r_1)^m
	f 
	\left(
		\left(
			\frac{
				1
			}{
				\operatorname{mes} Q_{2 R}^\varkappa
			}
			\int_{Q_{r_1}^\varkappa}
			|u|^p
			dx
			dt
		\right)^{1 / p}
	\right).
$$
To complete the proof, it remains to note that the last expression is equivalent to~\eqref{L3.2.2}.
\end{proof}

\begin{lemma}\label{L3.3}
Let $0 < p < 1$, then for all real numbers $R > 0$, $\varkappa > 0$, and $R \le r_1 < r_2 \le 2 R$ at least one of the following two inequalities is valid:
\begin{equation}
	J (r_2) - J (r_1)
	\ge
	C
	(r_2 - r_1)^k
	R^{k (\varkappa - 1)}
	f (J (r_1)),
	\label{L3.3.1}
\end{equation}
\begin{equation}
	J (r_2) - J (r_1)
	\ge
	C
	(r_2 - r_1)^{m / p}
	f^{1 / p} (J (r_1)),
	\label{L3.3.2}
\end{equation}
where
\begin{equation}
	J (r)
	=
	\frac{
		1
	}{
		\operatorname{mes} Q_{2 R}^\varkappa
	}
	\int_{Q_r^\varkappa}
	|u|
	dx
	dt,
	\quad
	R < r < 2 R.
	\label{L3.3.3}
\end{equation}
\end{lemma}

\begin{proof}
At first, assume that~\eqref{PL3.2.1} is valid.
As in the proof of Lemma~\ref{L3.2}, we can obviously assert that~\eqref{PL3.2.5} holds.
At the same time, since $f$ is a non-decreasing convex function, we have
\begin{align}
	\frac{1}{\operatorname{mes} Q_{2 R}^\varkappa}
	\int_{Q_{r_1}^\varkappa}
	f (|u|)
	dx
	dt
	&
	\ge
	\frac{
		1
	}{
		2^{n + \varkappa}
		\operatorname{mes} Q_{r_1}^\varkappa
	}
	\int_{Q_{r_1}^\varkappa}
	f (|u|)
	dx
	dt
	\nonumber
	\\
	&
	\ge
	\frac{
		1
	}{
		2^{n + \varkappa}
	}
	f 
	\left(
		\frac{
			1
		}{
			\operatorname{mes} Q_{r_1}^\varkappa
		}
		\int_{Q_{r_1}^\varkappa}
		|u|
		dx
		dt
	\right)
	\nonumber
	\\
	&
	\ge
	\frac{
		1
	}{
		2^{n + \varkappa}
	}
	f 
	\left(
		\frac{
			1
		}{
			\operatorname{mes} Q_{2 R}^\varkappa
		}
		\int_{Q_{r_1}^\varkappa}
		|u|
		dx
		dt
	\right).
	\label{PL3.3.1}
\end{align}
Thus,~\eqref{PL3.2.5} implies the estimate
$$
	\frac{1}{\operatorname{mes} Q_{2 R}^\varkappa}
	\int_{
		Q_{r_2}^\varkappa \setminus Q_{r_1}^\varkappa
	}
	|u|
	dx
	dt
	\ge
	C
	(r_2^\varkappa - r_1^\varkappa)^k
	f 
	\left(
		\frac{
			1
		}{
			\operatorname{mes} Q_{2 R}^\varkappa
		}
		\int_{Q_{r_1}^\varkappa}
		|u|
		dx
		dt
	\right),
$$
whence in accordance with~\eqref{PL3.2.4} inequality~\eqref{L3.3.1} follows at once.

In its turn, if~\eqref{PL3.2.1} is not valid, then Lemma~\ref{L3.1} implies~\eqref{PL3.2.6}, whence in accordance with~\eqref{PL3.2.7} and~\eqref{PL3.3.1} we obtain
\begin{equation}
	\frac{1}{\operatorname{mes} Q_{2 R}^\varkappa}
	\int_{Q_{r_2}^\varkappa \setminus Q_{r_1}^\varkappa}
	|u|^p
	dx
	dt
	\ge
	C
	(r_2 - r_1)^m
	f 
	\left(
		\frac{
			1
		}{
			\operatorname{mes} Q_{2 R}^\varkappa
		}
		\int_{Q_{r_1}^\varkappa}
		|u|
		dx
		dt
	\right).
	\label{PL3.3.2}
\end{equation}
By H\"older's inequality,
$$
	\int_{
		Q_{r_2}^\varkappa \setminus Q_{r_1}^\varkappa
	}
	|u|^p
	dx
	dt
	\le
	(\operatorname{mes} Q_{r_2}^\varkappa \setminus Q_{r_1}^\varkappa)^{1 - p}
	\left(
		\int_{
			Q_{r_2}^\varkappa \setminus Q_{r_1}^\varkappa
		}
		|u|
		dx
		dt
	\right)^p;
$$
therefore,~\eqref{PL3.3.2} yields
\begin{align*}
	\frac{1}{\operatorname{mes} Q_{2 R}^\varkappa}
	\int_{Q_{r_2}^\varkappa \setminus Q_{r_1}^\varkappa}
	|u|
	dx
	dt
	\ge
	{}
	&
	C
	(r_2 - r_1)^{m / p}
	\left(
		\frac{
			\operatorname{mes} Q_{2 R}^\varkappa
		}{
			\operatorname{mes} Q_{r_2}^\varkappa \setminus Q_{r_1}^\varkappa
		}
	\right)^{1 / p - 1}
	\\
	&
	{}
	\times
	f^{1 / p}
	\left(
		\frac{
			1
		}{
			\operatorname{mes} Q_{2 R}^\varkappa
		}
		\int_{Q_{r_1}^\varkappa}
		|u|
		dx
		dt
	\right).
\end{align*}
Since
$$
	\left(
		\frac{
			\operatorname{mes} Q_{2 R}^\varkappa
		}{
			\operatorname{mes} Q_{r_2}^\varkappa \setminus Q_{r_1}^\varkappa
		}
	\right)^{1 / p - 1}
	\ge
	1,
$$
this readily implies~\eqref{L3.3.2}.
\end{proof}

\begin{lemma}\label{L3.4}
Suppose that $p \ge 1$, then 
\begin{equation}
	F_\varkappa
	\left(
		\left(
			\frac{
				\sigma
			}{
				R^{n + \varkappa}
			}
			\int_{Q_R^\varkappa}
			|u|^p
			dx
			dt
		\right)^{1 / p}
	\right)
	\ge
	C R,
	\label{L3.4.1}
\end{equation}
for any real number $R > 0$ and $\varkappa > 0$ such that
$
	\operatorname{mes}
	\{
		(x, t) \in Q_R^\varkappa
		:
		u (x, t) \ne 0
	\}
	>
	0.
$
\end{lemma}

\begin{proof}
Take the minimal positive integer $l$ such that $2^l I (R) \ge I (2 R)$,
where the function $I$ is defined by~\eqref{L3.2.3}.
We construct a finite sequence of real numbers $r_i$, $i = 0, \ldots, l$, by putting
$r_0 = R$, 
$$
	r_i 
	=
	\sup 
	\{
		r \in (r_{i - 1}, 2 R)
		:
		I (r) \le 2 I (r_{i - 1})
	\},
	\quad
	0 < i \le l - 1,
$$
and  $r_l = 2 R$.
It can easily be seen that 
\begin{equation}
	I (r_i) = 2 I (r_{i-1})
	\quad
	\mbox{for all } 0 < i \le l - 1
	\quad
	\mbox{and}
	\quad
	I (r_l) \le 2 I (r_{l-1}).
	\label{PL3.4.6}
\end{equation}
According to Lemma~\ref{L3.2}, for any integer $0 < i \le l$ at least one of the following two inequalities is valid:
\begin{equation}
	I (r_i) - I (r_{i-1})
	\ge
	C
	(r_i - r_{i-1})^{k p}
	R^{k p (\varkappa - 1)}
	f^p (I^{1 / p} (r_{i-1})),
	\label{PL3.4.1}
\end{equation}
\begin{equation}
	I (r_i) - I (r_{i-1})
	\ge
	C
	(r_i - r_{i-1})^m
	f (I^{1 / p} (r_{i-1})).
	\label{PL3.4.2}
\end{equation}

We denote by $\Xi_1$ the set of all integers $0 < i \le l$ for which~\eqref{PL3.4.1} holds.
Also let $\Xi_2$ be the set of all other integers $0 < i \le l$.
If $i \in \Xi_1$, then~\eqref{PL3.4.6} and~\eqref{PL3.4.1} imply that
$$
	\frac{
		I^{1 / (k p)} (r_{i-1})
	}{
		f^{1 / k} (I^{1 / p} (r_{i-1}))
	}
	\ge
	C
	(r_i - r_{i-1})
	R^{\varkappa - 1},
$$
whence in accordance with the evident inequality
$$
	\int_{
		I (r_{i-1})
	}^{
		2 I (r_{i-1})
	}
	\frac{
		\zeta^{1 / (k p) - 1} 
		d\zeta
	}{
		f^{1 / k} (2^{- 1 / p} \zeta^{1 / p})
	}
	\ge
	\frac{
		2^{1 / (k p) - 1} 
		I^{1 / (k p)} (r_{i-1})
	}{
		f^{1 / k} (I^{1 / p} (r_{i-1}))
	}
$$
we obtain
\begin{equation}
	\int_{
		I (r_{i-1})
	}^{
		2 I (r_{i-1})
	}
	\frac{
		\zeta^{1 / (k p) - 1} 
		d\zeta
	}{
		f^{1 / k} (2^{- 1 / p} \zeta^{1 / p})
	}
	\ge
	C
	(r_i - r_{i-1})
	R^{\varkappa - 1},
	\label{PL3.4.3}
\end{equation}
In its turn, if $i \in \Xi_2$, then~\eqref{PL3.4.6} and~\eqref{PL3.4.2} allow us to assert that
$$
	\frac{
		I^{1 / m} (r_{i-1})
	}{
		f^{1 / m} (I^{1 / p} (r_{i-1}))
	}
	\ge
	C
	(r_i - r_{i-1}).
$$
Since
$$
	\int_{
		I (r_{i-1})
	}^{
		2 I (r_{i-1})
	}
	\frac{
		\zeta^{1 / m - 1}
		d\zeta
	}{
		f^{1 / m} (2^{- 1 / p} \zeta^{1 / p})
	}
	\ge
	\frac{
		2^{1 / m - 1}
		I^{1 / m} (r_{i-1})
	}{
		f^{1 / m} (I^{1 / p} (r_{i-1}))
	},
$$
this readily implies the estimate
\begin{equation}
	\int_{
		I (r_{i-1})
	}^{
		2 I (r_{i-1})
	}
	\frac{
		\zeta^{1 / m - 1}
		d\zeta
	}{
		f^{1 / m} (2^{- 1 / p} \zeta^{1 / p})
	}
	\ge
	C
	(r_i - r_{i-1}).
	\label{PL3.4.4}
\end{equation}

At first, we assume that
\begin{equation}
	\sum_{i \in \Xi_1}
	(r_i - r_{i-1})
	\ge
	\frac{R}{2}.
	\label{PL3.4.5}
\end{equation}
Summing~\eqref{PL3.4.3} over all $i \in \Xi_1$, we obtain
$$
	\int_{
		I (R)
	}^\infty
	\frac{
		\zeta^{1 / (k p) - 1} 
		d\zeta
	}{
		f^{1 / k} (2^{- 1 / p} \zeta^{1 / p})
	}
	\ge
	C
	R^\varkappa.
$$
By the change of variable $\xi = 2^{- 1 / p} \zeta^{1/p}$, the last expression is reduced to the form
$$
	\int_{
		(I (R) / 2)^{1 / p}
	}^\infty
	\frac{
		\xi^{1 / k - 1}
		d\xi
	}{
		f^{1 / k} (\xi)
	}
	\ge
	C
	R^\varkappa.
$$
This, in its turn, implies~\eqref{L3.4.1}.

Now let~\eqref{PL3.4.5} not hold. Then 
\begin{equation}
	\sum_{i \in \Xi_2}
	(r_i - r_{i-1})
	\ge
	\frac{R}{2};
	\label{PL3.4.7}
\end{equation}
therefore, summing~\eqref{PL3.4.4} over all $i \in \Xi_2$, we have
$$
	\int_{
		I (R)
	}^\infty
	\frac{
		\zeta^{1 / m - 1}
		d\zeta
	}{
	f^{1 / m} (2^{- 1 / p} \zeta^{1 / p})
	}
	\ge
	C R.
$$
By the change of variable $\xi = 2^{- 1 / p} \zeta^{1/p}$, this can be transformed to the inequality
$$
	\int_{
		(I (R) / 2)^{1 / p}
	}^\infty
	\frac{
		\xi^{p / m - 1}
		d\xi
	}{
		f^{1 / m} (\xi)
	}
	\ge
	C R,
$$
whence~\eqref{L3.4.1} follows again. The proof is completed.
\end{proof}

\begin{lemma}\label{L3.5}
Suppose that $0 < p < 1$, then 
\begin{equation}
	F_\varkappa
	\left(
		\frac{
			\sigma
		}{
			R^{n + \varkappa}
		}
		\int_{Q_R^\varkappa}
		|u|
		dx
		dt
	\right)
	\ge
	C R
	\label{L3.5.1}
\end{equation}
for any real number $R > 0$ and $\varkappa > 0$ such that
$
	\operatorname{mes}
	\{
		(x, t) \in Q_R^\varkappa
		:
		u (x, t) \ne 0
	\}
	>
	0.
$
\end{lemma}

\begin{proof}
We take the minimal positive integer $l$ such that $2^l J (R) \ge J (2 R)$, where the function $J$ is defined by~\eqref{L3.3.3}.
Consider the sequence of real numbers $r_i$, $i = 0, \ldots, l$, constructed in the proof of Lemma~\ref{L3.4} with $I$ replaced by $J$.
We obviously have 
\begin{equation}
	J (r_i) = 2 J (r_{i-1})
	\quad
	\mbox{for all } 0 < i \le l - 1
	\quad
	\mbox{and}
	\quad
	J (r_l) \le 2 J (r_{l-1}).
	\label{PL3.5.1}
\end{equation}
In view of Lemma~\ref{L3.3}, for any integer $0 < i \le l$ at least one of the following two inequalities is valid:
\begin{equation}
	J (r_i) - J (r_{i-1})
	\ge
	C
	(r_i - r_{i-1})^k
	R^{k (\varkappa - 1)}
	f (J (r_{i-1})),
	\label{PL3.5.2}
\end{equation}
\begin{equation}
	J (r_i) - J (r_{i-1})
	\ge
	C
	(r_i - r_{i-1})^{m / p}
	f^{1 / p} (J (r_{i-1})).
	\label{PL3.5.3}
\end{equation}

We denote by $\Xi_1$ the set of integers $0 < i \le l$ such that~\eqref{PL3.5.2} holds and let $\Xi_2$ be the set of all other positive integers $0 < i \le l$.
For any $i \in \Xi_1$, it follows from~\eqref{PL3.5.1} and~\eqref{PL3.5.2} that
$$
	\frac{
		J^{1 / k} (r_{i-1})
	}{
		f^{1 / k} (J (r_{i-1}))
	}
	\ge
	C
	(r_i - r_{i-1})
	R^{\varkappa - 1}.
$$
Combining this with the evident inequality
$$
	\int_{
		J (r_{i-1})
	}^{
		2 J (r_{i-1})
	}
	\frac{
		\zeta^{1 / k - 1}
		d\zeta
	}{
		f^{1 / k} (\zeta / 2)
	}
	\ge
	\frac{
		2^{1 / k - 1}
		J^{1 / k} (r_{i-1})
	}{
		f^{1 / k} (J (r_{i-1}))
	},
$$
we obtain
\begin{equation}
	\int_{
		J (r_{i-1})
	}^{
		2 J (r_{i-1})
	}
	\frac{
		\zeta^{1 / k - 1}
		d\zeta
	}{
		f^{1 / k} (\zeta / 2)
	}
	\ge
	C
	(r_i - r_{i-1})
	R^{\varkappa - 1}.
	\label{PL3.5.4}
\end{equation}
In its turn, for any $i \in \Xi_2$ in accordance with~\eqref{PL3.5.1} and~\eqref{PL3.5.3} we have
$$
	\frac{
		J^{p / m} (r_{i-1})
	}{
		f^{1 / m} (J (r_{i-1}))
	}
	\ge
	C
	(r_i - r_{i-1}).
$$
In view of the estimate
$$
	\int_{
		J (r_{i-1})
	}^{
		2 J (r_{i-1})
	}
	\frac{
		\zeta^{p / m - 1}
		d\zeta
	}{
		f^{1 / m} (\zeta / 2)
	}
	\ge
	\frac{
		2^{p / m - 1}
		J^{p / m} (r_{i-1})
	}{
		f^{1 / m} (J (r_{i-1}))
	},
$$
this implies that
\begin{equation}
	\int_{
		J (r_{i-1})
	}^{
		2 J (r_{i-1})
	}
	\frac{
		\zeta^{p / m - 1}
		d\zeta
	}{
		f^{1 / m} (\zeta / 2)
	}
	\ge
	C
	(r_i - r_{i-1}).
	\label{PL3.5.5}
\end{equation}

At first, let~\eqref{PL3.4.5} be fulfilled. Summing~\eqref{PL3.5.4} over all $i \in \Xi_1$, we obtain
$$
	\int_{
		J (R)
	}^\infty
	\frac{
		\zeta^{1 / k - 1}
		d\zeta
	}{
		f^{1 / k} (\zeta / 2)
	}
	\ge
	C
	R^\varkappa.
$$
After the change of variable $\xi = \zeta / 2$, the last expression takes the form
$$
	\int_{
		J (R) / 2
	}^\infty
	\frac{
		\xi^{1 / k - 1}
		d\xi
	}{
		f^{1 / k} (\xi)
	}
	\ge
	C
	R^\varkappa,
$$
whence~\eqref{L3.5.1} immediately follows.

Now, assume that~\eqref{PL3.4.5} is not valid. In this case,~\eqref{PL3.4.7} holds and, summing~\eqref{PL3.5.5} over all $i \in \Xi_2$, we arrive at the inequality
$$
	\int_{
		J (R)
	}^\infty
	\frac{
		\zeta^{p / m - 1}
		d\zeta
	}{
		f^{1 / m} (\zeta / 2)
	}
	\ge
	C
	R
$$
which, by the change of variable $\xi = \zeta / 2$, can be transformed to
$$
	\int_{
		J (R) / 2
	}^\infty
	\frac{
		\xi^{p / m - 1}
		d\xi
	}{
		f^{1 / m} (\xi)
	}
	\ge
	C
	R.
$$
This again leads us to~\eqref{L3.5.1}.
\end{proof}

\begin{lemma}\label{L3.6}
Suppose that $p \ge 1$, then 
\begin{align}
	\frac{
		1
	}{
		R^{k \varkappa - m}
	}
	\left(
		\frac{
			1
		}{
			\operatorname{mes} Q_R^\varkappa \setminus Q_{R / 2}^\varkappa
		}
		\int_{
			Q_R^\varkappa \setminus Q_{R / 2}^\varkappa
		}
		|u|^p
		dx
		dt
	\right)^{1 / p}
	&
	+
	\frac{
		1
	}{
		\operatorname{mes} Q_R^\varkappa \setminus Q_{R / 2}^\varkappa
	}
	\int_{
		Q_R^\varkappa \setminus Q_{R / 2}^\varkappa
	}
	|u|^p
	dx
	dt
	\nonumber
	\\
	&
	{}
	\ge
	\frac{
		C
	}{
		R^{n + \varkappa - m}
	}
	\int_{Q_{R / 2}^\varkappa}
	f (|u|)
	dx
	dt
	\label{PT2.1.3}
\end{align}
for all real numbers $R > 0$ and $\varkappa > 0$.
\end{lemma}

\begin{proof}
Lemma~\ref{L3.1} with $r_1 = R / 2$ and $r_2 = R$ yields
\begin{equation}
	\frac{
		1
	}{
		R^{k \varkappa}
	}
	\int_{
		Q_R^\varkappa \setminus Q_{R / 2}^\varkappa
	}
	|u|
	dx
	dt
	+
	\frac{
		1
	}{
		R^m
	}
	\int_{
		Q_R^\varkappa \setminus Q_{R / 2}^\varkappa
	}
	|u|^p
	dx
	dt
	\ge
	C
	\int_{
		Q_{R / 2}^\varkappa
	}
	f (|u|)
	dx
	dt
	\label{PT2.1.5}
\end{equation}
for all real numbers $R > 0$ and $\varkappa > 0$, whence in accordance with H\"older's inequality
$$
	(\operatorname{mes} Q_R^\varkappa \setminus Q_{R / 2}^\varkappa)^{(p - 1) / p}
	\left(
		\int_{
			Q_R^\varkappa \setminus Q_{R / 2}^\varkappa
		}
		|u|^p
		dx
		dt
	\right)^{1 / p}
	\ge
	\int_{
		Q_R^\varkappa \setminus Q_{R / 2}^\varkappa
	}
	|u|
	dx
	dt
$$
and the fact that
\begin{equation}
	\sigma R^{n + \varkappa} 
	\le 
	\operatorname{mes} Q_R^\varkappa \setminus Q_{R / 2}^\varkappa
	\le 
	C R^{n + \varkappa}
	\label{PL3.6.1}
\end{equation}
estimate~\eqref{PT2.1.3} follows at once.
\end{proof}

\begin{lemma}\label{L3.7}
Assume that $0 < p < 1$, then
\begin{align}
	\frac{
		1
	}{
		\operatorname{mes} Q_R^\varkappa \setminus Q_{R / 2}^\varkappa
	}
	\int_{
		Q_R^\varkappa \setminus Q_{R / 2}^\varkappa
	}
	|u|
	dx
	dt
	&
	+
	\frac{
		1
	}{
		R^{m - k \varkappa}
	}
	\left(
		\frac{
			1
		}{
			\operatorname{mes} Q_R^\varkappa \setminus Q_{R / 2}^\varkappa
		}
		\int_{
			Q_R^\varkappa \setminus Q_{R / 2}^\varkappa
		}
		|u|
		dx
		dt
	\right)^p
	\nonumber
	\\
	&
	{}
	\ge
	\frac{
		C
	}{
		R^{n - (k - 1) \varkappa}
	}
	\int_{Q_{R / 2}^\varkappa}
	f (|u|)
	dx
	dt
	\label{PT2.2.3}
\end{align}
for all real numbers $R > 0$ and $\varkappa > 0$.
\end{lemma}

\begin{proof}
In view of Lemma~\ref{L3.1}, relation~\eqref{PT2.1.5} holds for all real numbers $R > 0$ and $\varkappa > 0$. 
Evaluating the second summand on the left in~\eqref{PT2.1.5} by H\"older's inequality 
$$
	(\operatorname{mes} Q_R^\varkappa \setminus Q_{R / 2}^\varkappa)^{1 - p}
	\left(
		\int_{
			Q_R^\varkappa \setminus Q_{R / 2}^\varkappa
		}
		|u|
		dx
		dt
	\right)^p
	\ge
	\int_{
		Q_R^\varkappa \setminus Q_{R / 2}^\varkappa
	}
	|u|^p
	dx
	dt
$$
and using estimate~\eqref{PL3.6.1}, we readily obtain~\eqref{PT2.2.3}.
\end{proof}

\begin{lemma}\label{L3.8}
Let $p \ge 1$, then for all real numbers $R > 0$ and $\varkappa > 0$ 
at least one of the following two systems of inequalities is valid:
\begin{equation}
	\left\{
		\begin{aligned}
			&
			\left(
				\frac{
					1
				}{
					\operatorname{mes} Q_R^\varkappa \setminus Q_{R / 2}^\varkappa
				}
				\int_{
					Q_R^\varkappa \setminus Q_{R / 2}^\varkappa
				}
				|u|^p
				dx	
				dt
			\right)^{1 / p}
			\ge
			\frac{
				\sigma
			}{
			R^{n - (k - 1) \varkappa}
			}
			\int_{Q_{R / 2}^\varkappa}
			f (|u|)
			dx
			dt,
			\\
			&
			E (R) - E (R / 2)
			\ge
			C
			R^{n + \varkappa}
			f 
			\left(
				\frac{
					\sigma E (R / 2)
				}{
					R^{n - (k - 1) \varkappa}
				}
			\right),
		\end{aligned}
	\right.
	\label{L3.8.1}
\end{equation}
\begin{equation}
	\left\{
		\begin{aligned}
			&
			\frac{
				1
			}{
				\operatorname{mes} Q_R^\varkappa \setminus Q_{R / 2}^\varkappa
			}
			\int_{
				Q_R^\varkappa \setminus Q_{R / 2}^\varkappa
			}
			|u|^p
			dx
			dt
			\ge
			\frac{
				\sigma
			}{
				R^{n + \varkappa - m}
			}
			\int_{Q_{R / 2}^\varkappa}
			f (|u|)
			dx
			dt,
			\\
			&
			E (R) - E (R / 2)
			\ge
			C
			R^{n + \varkappa}
			f 
			\left(
				\left(
					\frac{
						\sigma E (R / 2)
					}{
						R^{n + \varkappa - m}
					}
				\right)^{1 / p}
			\right),
		\end{aligned}
	\right.
	\label{L3.8.2}
\end{equation}
where
\begin{equation}
	E (r)
	=
	\int_{Q_r^\varkappa}
	f (|u|)
	dx
	dt,
	\quad
	r > 0.
	\label{E}
\end{equation}
\end{lemma}

\begin{proof}
Let
\begin{align}
	&
	\frac{
		1
	}{
		R^{k \varkappa - m}
	}
	\left(
		\frac{
			1
		}{
			\operatorname{mes} Q_R^\varkappa \setminus Q_{R / 2}^\varkappa
		}
		\int_{
			Q_R^\varkappa \setminus Q_{R / 2}^\varkappa
		}
		|u|^p
		dx
		dt
	\right)^{1 / p}
	\nonumber
	\\
	&
	\qquad
	{}
	\ge
	\frac{
		1
	}{
		\operatorname{mes} Q_R^\varkappa \setminus Q_{R / 2}^\varkappa
	}
	\int_{
		Q_R^\varkappa \setminus Q_{R / 2}^\varkappa
	}
	|u|^p
	dx
	dt.
	\label{PL3.8.1}
\end{align}
Applying Lemma~\ref{L3.6}, we have the first inequality in~\eqref{L3.8.1}, whence it follows that
$$
	f
	\left(
		\left(
			\frac{
				1
			}{
				\operatorname{mes} Q_R^\varkappa \setminus Q_{R / 2}^\varkappa
			}
			\int_{
				Q_R^\varkappa \setminus Q_{R / 2}^\varkappa
			}
			|u|^p
			dx
			dt
		\right)^{1 / p}
	\right)
	\ge
	f
	\left(
		\frac{
			\sigma
		}{
			R^{n - (k - 1) \varkappa}
		}
		\int_{Q_{R / 2}^\varkappa}
		f (|u|)
		dx
		dt
	\right)
$$
since $f$ is a non-decreasing function.
Combining this with the estimate
\begin{align}
	&
	\frac{
		1
	}{
		\operatorname{mes} Q_R^\varkappa \setminus Q_{R / 2}^\varkappa
	}
	\int_{
		Q_R^\varkappa
		\setminus
		Q_{R / 2}^\varkappa
	}
	f (|u|)
	dx
	dt
	\nonumber
	\\
	&
	\qquad
	{}
	\ge
	f
	\left(
		\left(
			\frac{
				1
			}{
				\operatorname{mes} Q_R^\varkappa \setminus Q_{R / 2}^\varkappa
			}
			\int_{
				Q_R^\varkappa
				\setminus
				Q_{R / 2}^\varkappa
			}
			|u|^p
			dx
			dt
		\right)^{1 / p}
	\right)
	\label{PL3.8.2}
\end{align}
which is valid as $f (\zeta^{1/p})$ is a convex function, we obtain
$$
	\frac{
		1
	}{
		\operatorname{mes} Q_R^\varkappa \setminus Q_{R / 2}^\varkappa
	}
	\int_{
		Q_R^\varkappa
		\setminus
		Q_{R / 2}^\varkappa
	}
	f (|u|)
	dx
	dt
	\ge
	f
	\left(
		\frac{
			\sigma
		}{
			R^{n - (k - 1) \varkappa}
		}
		\int_{Q_{R / 2}^\varkappa}
		f (|u|)
		dx
		dt
	\right).
$$
In view of~\eqref{PL3.6.1}, the last expression is equivalent to the second inequality in~\eqref{L3.8.1}.

Now, assume that~\eqref{PL3.8.1} is not satisfied. In this case, Lemma~\ref{L3.6} implies the fist inequality in~\eqref{L3.8.2} which, in turn, allows us to assert that
\begin{align*}
	&
	f
	\left(
		\left(
			\frac{
				1
			}{
				\operatorname{mes} Q_R^\varkappa \setminus Q_{R / 2}^\varkappa
			}
			\int_{
				Q_R^\varkappa \setminus Q_{R / 2}^\varkappa
			}
			|u|^p
			dx
			dt
		\right)^{1 / p}
	\right)
	\\
	&
	\qquad
	{}
	\ge
	f
	\left(
		\left(
			\frac{
				\sigma
			}{
				R^{n + \varkappa - m}
			}
			\int_{Q_{R / 2}^\varkappa}
			f (|u|)
			dx
			dt
		\right)^{1 / p}
	\right).
\end{align*}
Combining this with~\eqref{PL3.8.2}, we have
$$
	\frac{
		1
	}{
		\operatorname{mes} Q_R^\varkappa \setminus Q_{R / 2}^\varkappa
	}
	\int_{
		Q_R^\varkappa
		\setminus
		Q_{R / 2}^\varkappa
	}
	f (|u|)
	dx
	dt
	\ge
	f
	\left(
		\left(
			\frac{
				\sigma
			}{
				R^{n + \varkappa - m}
			}
			\int_{Q_{R / 2}^\varkappa}
			f (|u|)
			dx
			dt
		\right)^{1 / p}
	\right),
$$
whence in accordance with~\eqref{PL3.6.1} the second inequality in~\eqref{L3.8.2} follows at once.
\end{proof}

\begin{lemma}\label{L3.9}
Let $0 < p < 1$, then for all real numbers $R > 0$ and $\varkappa > 0$ at least one of the following two systems of inequalities is valid:
$$
	\left\{
		\begin{aligned}
			&
			\frac{
				1
			}{
				\operatorname{mes} Q_R^\varkappa \setminus Q_{R / 2}^\varkappa
			}
			\int_{
				Q_R^\varkappa \setminus Q_{R / 2}^\varkappa
			}
			|u|
			dx
			dt
			\ge
			\frac{
				\sigma
			}{
				R^{n - (k - 1) \varkappa}
			}
			\int_{Q_{R / 2}^\varkappa}
			f (|u|)
			dx
			dt,
			\\
			&
			E (R) - E (R / 2)
			\ge
			C
			R^{n + \varkappa}
			f 
			\left(
				\frac{
					\sigma E (R / 2)
				}{
					R^{n - (k - 1) \varkappa}
				}
			\right),
		\end{aligned}
	\right.
$$
$$
	\left\{
		\begin{aligned}
			&
			\left(
				\frac{
					1
				}{
					\operatorname{mes} Q_R^\varkappa \setminus Q_{R / 2}^\varkappa
				}
				\int_{
					Q_R^\varkappa \setminus Q_{R / 2}^\varkappa
				}
				|u|
				dx
				dt
			\right)^p
			\ge
			\frac{
				\sigma
			}{
				R^{n + \varkappa - m}
			}
			\int_{Q_{R / 2}^\varkappa}
			f (|u|)
			dx
			dt,
			\\
			&
			E (R) - E (R / 2)
			\ge
			C
			R^{n + \varkappa}
			f 
			\left(
				\left(
					\frac{
						\sigma E (R / 2)
					}{
						R^{n + \varkappa - m}
					}
				\right)^{1 / p}
			\right),
		\end{aligned}
	\right.
$$
where the function $E$ is defined by~\eqref{E}.
\end{lemma}

\begin{proof}
The proof is completely analogous to the proof of Lemma~\ref{L3.8}. The only difference is that Lemma~\ref{L3.6} should be replaced by Lemma~\ref{L3.7}.
\end{proof}

\section{Proof of Theorems~\ref{T2.0}--\ref{T2.3}}\label{proof}

In this section, as in the previous one, by $C$ and $\sigma$ we denote various positive constants that can depend only on $A$, $k$, $m$, $n$, and $p$, unless otherwise stated.

\begin{proof}[Proof of Theorem~\ref{T2.0}]
Let us take $\varkappa = m / k$.
At first, we consider the case where $p \ge 1$. For any non-trivial global weak solution of~\eqref{1.1i}, \eqref{1.1c} in accordance with the evident estimate 
$$
	\sigma R^{n + \varkappa} 
	\le 
	\operatorname{mes} Q_R^\varkappa
	\le 
	C R^{n + \varkappa}
$$
Lemma~\ref{L3.4} implies that
$$
	\lim_{R \to \infty}
	\frac{
		1
	}{
		\operatorname{mes} Q_R^\varkappa
	}
	\int_{Q_R^\varkappa}
	|u|^p
	dx
	dt
	=
	0.
$$
In do doing, Lemma~\ref{L3.6} yields
\begin{align*}
	\left(
		\frac{
			1
		}{
			\operatorname{mes} Q_R^\varkappa \setminus Q_{R / 2}^\varkappa
		}
		\int_{
			Q_R^\varkappa \setminus Q_{R / 2}^\varkappa
		}
		|u|^p
		dx
		dt
	\right)^{1 / p}
	&
	+
	\frac{
		1
	}{
		\operatorname{mes} Q_R^\varkappa \setminus Q_{R / 2}^\varkappa
	}
	\int_{
		Q_R^\varkappa \setminus Q_{R / 2}^\varkappa
	}
	|u|^p
	dx
	dt
	\\
	&
	{}
	\ge
	\frac{
		C
	}{
		R^{n - m (1 - 1 / k)}
	}
	\int_{Q_{R / 2}^\varkappa}
	f (|u|)
	dx
	dt
\end{align*}
for all real numbers $R > 0$. Passing in the last expression to the limit as $R \to \infty$, we have
$$
	\int_{
		{\mathbb R}_+^{n+1}
	}
	f (|u|)
	dx
	dt
	=
	0.
$$
Since $f (\zeta) > 0$ for all $\zeta \in (0, \infty)$, this implies that $u$ is equal to zero almost everywhere in ${\mathbb R}_+^{n+1}$. Thus, we arrive at a contradiction with our assumption that $u$ is a non-trivial solution.

For $0 < p < 1$, all the arguments are completely analogous. The only difference is that Lemmas~\ref{L3.4} and~\ref{L3.6} should be replaced by  Lemmas~\ref{L3.5} and~\ref{L3.7}, respectively.
\end{proof}

\begin{proof}[Proof of Theorem~\ref{T2.1}]
Assume the converse. 
Let $u$ be a non-trivial global weak solution of~\eqref{1.1i}, \eqref{1.1c}.
We put
\begin{equation}
	\varkappa
	=
	\frac{
		n (p - 1) + m
	}{
		1 + (k - 1) p
	}.
	\label{varkappa}
\end{equation}
Since $p > 1 - m / n$ by the hypotheses of the theorem, one can assert that $\varkappa > 0$.
It can also be verified that $\varkappa$ satisfies the system of equations
\begin{equation}
	\left\{
		\begin{aligned}
			&
			n + \varkappa - \gamma (n - (k - 1) \varkappa) = 0,
			\\
			&
			n + \varkappa - \frac{\gamma}{p} (n + \varkappa - m) = 0,
		\end{aligned}
		\label{PT2.2.1new}
	\right.
\end{equation}
where $\gamma$ is defined by~\eqref{T2.1.2}. 
In particular, from~\eqref{PT2.2.1new}, it follows that
\begin{equation}
	\frac{
		n + \varkappa - m
	}{
		p
	}
	=
	n - (k - 1) \varkappa.
	\label{connection}
\end{equation}

We denote
\begin{equation}
	h (\zeta)
	=
	\frac{
		f (\zeta)
	}{
		\zeta^\gamma
	},
	\quad
	\zeta > 0.
	\label{h}
\end{equation}
Take a real number $r_0 > 0$ such that $E (r_0) > 0$, where the function $E$ is defined by~\eqref{E}.
Also let $r_i = 2^i r_0$, $i = 1,2,\ldots$.

From now on, we assume that the constants $C$ and $\sigma$ can also depend on $E (r_0)$. 

At first, we consider the case of $p \ge 1$.
Let us establish the validity of the estimate
\begin{equation}
	\frac{
		E (r_{i+1}) - E (r_i)
	}{
		E^{\gamma / p} (r_i)
	}
	\ge
	C
	h 
	\left(
		\frac{
			\sigma E^{1 / p} (r_i)
		}{
			r_i^{n - (k - 1) \varkappa}
		}
	\right)
	\label{PT2.2.4new}
\end{equation}
for all $i = 0,1,2,\ldots$. 
Indeed, by Lemma~\ref{L3.8}, for any integer $i \ge 0$ at least one of the following two systems of inequalities is satisfied:
\begin{equation}
	\left\{
		\begin{aligned}
			&
			\left(
				\frac{
					1
				}{
					\operatorname{mes} Q_{r_{i+1}}^\varkappa \setminus Q_{r_i}^\varkappa
				}
				\int_{
					Q_{r_{i+1}}^\varkappa \setminus Q_{r_i}^\varkappa
				}
				|u|^p
				dx	
				dt
			\right)^{1 / p}
			\ge
			\frac{
				\sigma
			}{
			r_i^{n - (k - 1) \varkappa}
			}
			\int_{Q_{r_i}^\varkappa}
			f (|u|)
			dx
			dt,
			\\
			&
			E (r_{i+1}) - E (r_i)
			\ge
			C
			r_i^{n + \varkappa}
			f 
			\left(
				\frac{
					\sigma E (r_i)
				}{
					r_i^{n - (k - 1) \varkappa}
				}
			\right),
		\end{aligned}
	\right.
	\label{PT2.2.2new}
\end{equation}
\begin{equation}
	\left\{
		\begin{aligned}
			&
			\frac{
				1
			}{
				\operatorname{mes} Q_{r_{i+1}}^\varkappa \setminus Q_{r_i}^\varkappa
			}
			\int_{
				Q_{r_{i+1}}^\varkappa \setminus Q_{r_i}^\varkappa
			}
			|u|^p
			dx
			dt
			\ge
			\frac{
				\sigma
			}{
				r_i^{n + \varkappa - m}
			}
			\int_{Q_{r_i}^\varkappa}
			f (|u|)
			dx
			dt,
			\\
			&
			E (r_{i+1}) - E (r_i)
			\ge
			C
			r_i^{n + \varkappa}
			f 
			\left(
				\left(
					\frac{
						\sigma E (r_i)
					}{
						r_i^{n + \varkappa - m}
					}
				\right)^{1 / p}
			\right).
		\end{aligned}
	\right.
	\label{PT2.2.3new}
\end{equation}

Let~\eqref{PT2.2.2new} be valid for some  integer $i \ge 0$.
Since $f$ is a non-decreasing function, the second inequality in~\eqref{PT2.2.2new} implies that
$$
	E (r_{i+1}) - E (r_i)
	\ge
	C
	r_i^{n + \varkappa}
	f 
	\left(
		\frac{
			\sigma E^{1 - 1 / p} (r_0) E^{1 / p} (r_i)
		}{
			r_i^{n - (k - 1) \varkappa}
		}
	\right),
$$
whence in accordance with the first equation in~\eqref{PT2.2.1new}, we readily obtain~\eqref{PT2.2.4new}.

In its turn, if~\eqref{PT2.2.3new} is valid for some integer $i \ge 0$, then the second inequality in~\eqref{PT2.2.3new} yields
$$
	E (r_{i+1}) - E (r_i)
	\ge
	C
	r_i^{n + \varkappa}
	f 
	\left(
		\frac{
			\sigma E^{1 / p} (r_i)
		}{
			r_i^{(n + \varkappa - m) / p}
		}
	\right),
$$
whence in accordance with~\eqref{connection} and the second equation in~\eqref{PT2.2.1new} we have~\eqref{PT2.2.4new}.

It can be seen that
\begin{equation}
	\left(
		\frac{
			1
		}{
			\operatorname{mes} Q_{r_{i+1}}^\varkappa \setminus Q_{r_i}^\varkappa
		}
		\int_{
			Q_{r_{i+1}}^\varkappa \setminus Q_{r_i}^\varkappa
		}
		|u|^p
		dx	
		dt
	\right)^{1 / p}
	\ge	
	\frac{
		\sigma E^{1 / p} (r_i)
	}{
		r_i^{n - (k - 1) \varkappa}
	}
	\label{PT2.2.5new}
\end{equation}
for all $i = 0,1,2,\ldots$.
Indeed, in the case where~\eqref{PT2.2.2new} holds, estimate~\eqref{PT2.2.5new} follows from the first inequality in~\eqref{PT2.2.2new} and the fact that 
$$
	\int_{Q_{r_i}^\varkappa}
	f (|u|)
	dx
	dt
	\ge 
	E^{1 - 1 / p} (r_0) E^{1 / p} (r_i).
$$
In its turn, if~\eqref{PT2.2.3new} holds, then~\eqref{PT2.2.5new} follows from~\eqref{connection} and the first inequality in~\eqref{PT2.2.3new}.

By Lemma~\ref{L3.4} and estimate~\eqref{PL3.6.1}, the left-hand side of~\eqref{PT2.2.5new} tends to zero as $i \to \infty$. Thus, passing in~\eqref{PT2.2.5new} to the limit as $i \to \infty$, we obtain
\begin{equation}
	\lim_{i \to \infty}
	\frac{
		E^{1 / p} (r_i)
	}{
		r_i^{n - (k - 1) \varkappa}
	}
	=
	0.
	\label{new}
\end{equation}

Let us denote
$
	\lambda = n - (k - 1) \varkappa
$
and
$$
	\zeta_i = E^{1 / p} (r_i), 
	\quad
	i = 0,1,2,\ldots.
$$
Taking into account~\eqref{varkappa} and the condition $n > m (1 - 1 / k)$, we have
$$
	\lambda
	=
	\frac{
		k (n - m (1 - 1 / k))
	}{
		1 + (k - 1) p
	}
	>
	0.
$$
There are sequences of integers $0 \le s_i < l_i \le s_{i+1}$, $i = 1,2,\ldots$, such that
$$
	\frac{
		\zeta_j
	}{
		r_j^\lambda
	}
	>
	\frac{
		\zeta_{j+1}
	}{
		r_{j+1}^\lambda
	}
$$
if 
\begin{equation}
	j 
	\in 
	\bigcup_{i=1}^\infty [s_i, l_i)
	\label{newPT2.1.1}
\end{equation}
and
$$
	\frac{
		\zeta_j
	}{
		r_j^\lambda
	}
	\le
	\frac{
		\zeta_{j+1}
	}{
		r_{j+1}^\lambda
	}
$$
if~\eqref{newPT2.1.1} is not fulfilled.
Since $E (r_{j+1}) \ge E (r_j)$ for all $j = 1,2,\ldots$, we have
\begin{equation}
	\frac{
		2^\lambda
		\zeta_{j+1}
	}{
		r_{j+1}^\lambda
	}
	\ge
	\frac{
		\zeta_j
	}{
		r_j^\lambda
	}
	>
	\frac{
		\zeta_{j+1}
	}{
		r_{j+1}^\lambda
	}
	\label{newPT2.1.2}
\end{equation}
for all integers $j$ satisfying~\eqref{newPT2.1.1}.
It can also be seen that
\begin{equation}
	2^\lambda \zeta_j > \zeta_{j+1} \ge \zeta_j
	\label{newPT2.1.3}
\end{equation}
for all integers $j$ satisfying~\eqref{newPT2.1.1}.

Estimate~\eqref{PT2.2.4new} can obviously be written as
\begin{equation}
	\frac{
		\zeta_{j+1}^p - \zeta_j^p
	}{
		\zeta_j^\gamma
	}
	\ge
	C
	h 
	\left(
		\frac{
			\sigma \zeta_j
		}{
			r_j^\lambda
		}
	\right),
	\quad
	j = 1,2,\ldots,
	\label{PT2.2.12new}
\end{equation}
whence in accordance with~\eqref{newPT2.1.3} it follows that
$$
	\frac{
		\zeta_{j+1} - \zeta_j
	}{
		\zeta_j^{\gamma - p + 1}
	}
	\ge
	C
	h 
	\left(
		\frac{
			\sigma \zeta_j
		}{
			r_j^\lambda
		}
	\right)
$$
for all integers $j$ satisfying~\eqref{newPT2.1.1}.
Multiplying this by the inequality
$$
	1 
	\ge
	\frac{
		\zeta_j / r_j^\lambda - \zeta_{j+1} / r_{j+1}^\lambda
	}{	
		\zeta_j / r_j^\lambda
	},
$$
we obtain
\begin{equation}
	\frac{
		\zeta_{j+1} - \zeta_j
	}{
		\zeta_j^{\gamma - p + 1}
	}
	\ge
	C
	\frac{
		h (\sigma \zeta_j / r_j^\lambda)
	}{
		\zeta_j / r_j^\lambda
	}
	\left(
		\frac{
			\zeta_j
		}{
			r_j^\lambda
		}
		- 
		\frac{
			\zeta_{j+1}
		}{
			r_{j+1}^\lambda
		}
	\right)
	\label{newPT2.1.4}
\end{equation}
for all integers $j$ satisfying~\eqref{newPT2.1.1}. 

In view of~\eqref{newPT2.1.3}, we have
$$
	\int_{
		\zeta_j
	}^{
		\zeta_{j+1}
	}
	\frac{
		d\zeta
	}{
		\zeta^{
			\gamma - p + 1
		}
	}
	\ge
	C
	\frac{
		\zeta_{j+1} - \zeta_j
	}{
		\zeta_j^{\gamma - p + 1}
	}
$$
for all integers $j$ satisfying~\eqref{newPT2.1.1}.
In turn,~\eqref{newPT2.1.2} allows us to assert that
$$
	\frac{
		h (\sigma \zeta_j / r_j^\lambda)
	}{
		\zeta_j / r_j^\lambda
	}
	\left(
		\frac{
			\zeta_j
		}{
			r_j^\lambda
		}
		- 
		\frac{
			\zeta_{j+1}
		}{
			r_{j+1}^\lambda
		}
	\right)
	\ge
	C
	\int_{
		\zeta_{j+1} / r_{j+1}^\lambda
	}^{
		\zeta_j / r_j^\lambda
	}
	\frac{
		\tilde h (\sigma \zeta)
	}{
		\zeta
	}
	d\zeta
$$
for all integers $j$ satisfying~\eqref{newPT2.1.1}, where
$$
	\tilde h (\zeta)
	=
	\inf_{
		(
			\zeta, 
			\, 
			2^\lambda
			\zeta
		)
	}
	h.
$$
Thus,~\eqref{newPT2.1.4} implies the estimate
$$
	\int_{
		\zeta_j
	}^{
		\zeta_{j+1}
	}
	\frac{
		d\zeta
	}{
		\zeta^{
			\gamma - p + 1
		}
	}
	\ge
	C
	\int_{
		\zeta_{j+1} / r_{j+1}^\lambda
	}^{
		\zeta_j / r_j^\lambda
	}
	\frac{
		\tilde h (\sigma \zeta)
	}{
		\zeta
	}
	d\zeta
$$
for all integers $j$ satisfying~\eqref{newPT2.1.1}, whence it follows that
$$
	\sum_{i=1}^\infty
	\int_{
		\zeta_{s_i}
	}^{
		\zeta_{l_i}
	}
	\frac{
		d\zeta
	}{
		\zeta^{
			\gamma - p + 1
		}
	}
	\ge
	C
	\sum_{i=1}^\infty
	\int_{
		\zeta_{l_i} / r_{l_i}^\lambda  
	}^{
		\zeta_{s_i} / r_{s_i}^\lambda
	}
	\frac{
		\tilde h (\sigma \zeta)
	}{
		\zeta
	}
	d\zeta.
$$
Combining this with~\eqref{new} and the inequalities
$
	\zeta_{s_{i+1}} \ge \zeta_{l_i}
$
and
$
	\zeta_{l_i} / r_{l_i}^\lambda  
	\le
	\zeta_{s_{i+1}} /	r_{s_{i+1}}^\lambda
$,
$i = 1,2,\ldots,$
we obtain
\begin{equation}
	\int_{
		\zeta_{s_1}
	}^\infty
	\frac{
		d\zeta
	}{
		\zeta^{
			\gamma - p + 1
		}
	}
	\ge
	C
	\int_0^{
		\zeta_{s_1} / r_{s_1}^\lambda
	}
	\frac{
		\tilde h (\sigma \zeta)
	}{
		\zeta
	}
	d\zeta.
	\label{newPT2.1.6}
\end{equation}
In view of the choice of the real number $r_0$, one can assert that $E (r_{s_1}) \ge E (r_0) > 0$.
At the same time, $\gamma > p$ in accordance with~\eqref{T2.1.2}. Hence,
$$
	\int_{
		\zeta_{s_1}
	}^\infty
	\frac{
		d\zeta
	}{
		\zeta^{
			\gamma - p + 1
		}
	}
	<
	\infty.
$$
In so doing, from the monotonicity of the function $f$, it follows that
\begin{equation}
	\tilde h (\zeta)
	\ge
	\frac{
		1
	}{
		2^{\lambda \gamma}
	}
	\frac{
		f (\zeta)
	}{
		\zeta^{
			\gamma
		}
	},
	\quad
	\zeta > 0.
	\label{PT2.2.13new}
\end{equation}
Thus,~\eqref{newPT2.1.6} implies the relation
$$
	\int_0^{
		\zeta_{s_1} / r_{s_1}^\lambda
	}
	\frac{
		f (\sigma \zeta)
	}{
		\zeta^{
			1 + \gamma
		}
	}
	d\zeta
	<
	\infty
$$
which contradicts~\eqref{T2.1.1}. 

Now, we consider the case of $1 - m / n < p < 1$. Let us show that
\begin{equation}
	\frac{
		E (r_{i+1}) - E (r_i)
	}{
		E^\gamma (r_i)
	}
	\ge
	C
	h 
	\left(
		\frac{
			\sigma E (r_i)
		}{
			r_i^{n - (k - 1) \varkappa}
		}
	\right)
	\label{PT2.2.6new}
\end{equation}
for all $i = 0,1,2,\ldots$. 
Indeed, by Lemma~\ref{L3.9}, for any integer $i \ge 0$ at least one of the following two systems of inequalities is satisfied:
\begin{equation}
	\left\{
		\begin{aligned}
			&
			\frac{
				1
			}{
				\operatorname{mes} Q_{r_{i+1}}^\varkappa \setminus Q_{r_i}^\varkappa
			}
			\int_{
				Q_{r_{i+1}}^\varkappa \setminus Q_{r_i}^\varkappa
			}
			|u|
			dx
			dt
			\ge
			\frac{
				\sigma
			}{
				r_i^{n - (k - 1) \varkappa}
			}
			\int_{Q_{r_i}^\varkappa}
			f (|u|)
			dx
			dt,
			\\
			&
			E (r_{i+1}) - E (r_i)
			\ge
			C
			r_i^{n + \varkappa}
			f 
			\left(
				\frac{
					\sigma E (r_i)
				}{
					r_i^{n - (k - 1) \varkappa}
				}
			\right),
		\end{aligned}
	\right.
	\label{PT2.2.7new}
\end{equation}
\begin{equation}
	\left\{
		\begin{aligned}
			&
			\left(
				\frac{
					1
				}{
					\operatorname{mes} Q_{r_{i+1}}^\varkappa \setminus Q_{r_i}^\varkappa
				}
				\int_{
					Q_{r_{i+1}}^\varkappa \setminus Q_{r_i}^\varkappa
				}
				|u|
				dx
				dt
			\right)^p
			\ge
			\frac{
				\sigma
			}{
				r_i^{n + \varkappa - m}
			}
			\int_{Q_{r_i}^\varkappa}
			f (|u|)
			dx
			dt,
			\\
			&
			E (r_{i+1}) - E (r_i)
			\ge
			C
			r_i^{n + \varkappa}
			f 
			\left(
				\left(
					\frac{
						\sigma E (r_i)
					}{
						r_i^{n + \varkappa - m}
					}
				\right)^{1 / p}
			\right).
		\end{aligned}
	\right.
	\label{PT2.2.8new}
\end{equation}
In the case where~\eqref{PT2.2.7new} is valid for some integer $i \ge 0$, the second inequality in~\eqref{PT2.2.7new} and the first equation in~\eqref{PT2.2.1new} immediately imply~\eqref{PT2.2.6new}.
In its turn, if~\eqref{PT2.2.8new} is valid for some integer $i \ge 0$, then~\eqref{connection} and the second inequality in~\eqref{PT2.2.8new} yield
$$
	E (r_{i+1}) - E (r_i)
	\ge
	C
	r_i^{n + \varkappa}
	f 
	\left(
		\frac{
			\sigma E^{1 / p} (r_i)
		}{
			r_i^{n - (k - 1) \varkappa}
		}
	\right),
$$
whence in accordance with the monotonicity of the function $f$ and the fact that 
\begin{equation}
	E^{1 / p} (r_i) 
	\ge 
	E^{1 / p - 1} (r_0) 
	E (r_i)
	\label{PT2.2.9new}
\end{equation}
we obtain
$$
	E (r_{i+1}) - E (r_i)
	\ge
	C
	r_i^{n + \varkappa}
	f 
	\left(
		\frac{
			\sigma E (r_i)
		}{
			r_i^{n - (k - 1) \varkappa}
		}
	\right).
$$
According to the first equation in~\eqref{PT2.2.1new} and the definition of function $h$, this again leads us to~\eqref{PT2.2.6new}.

It can also be shown that
\begin{equation}
	\frac{
		1
	}{
		\operatorname{mes} Q_{r_{i+1}}^\varkappa \setminus Q_{r_i}^\varkappa
	}
	\int_{
		Q_{r_{i+1}}^\varkappa \setminus Q_{r_i}^\varkappa
	}
	|u|
	dx
	dt
	\ge
	\frac{
		\sigma
		E (r_i)
	}{
		r_i^{n - (k - 1) \varkappa}
	}
	\label{PT2.2.10new}
\end{equation}
for all $i = 0,1,2,\ldots$.
Indeed, if~\eqref{PT2.2.7new} is satisfied for some integer $i \ge 0$, then~\eqref{PT2.2.10new} follows directly from the first inequality in~\eqref{PT2.2.7new}.
In the case where~\eqref{PT2.2.8new} holds, the first inequality in~\eqref{PT2.2.8new} implies that
$$
	\frac{
		1
	}{
		\operatorname{mes} Q_{r_{i+1}}^\varkappa \setminus Q_{r_i}^\varkappa
	}
	\int_{
		Q_{r_{i+1}}^\varkappa \setminus Q_{r_i}^\varkappa
	}
	|u|
	dx
	dt
	\ge
	\frac{
		\sigma
		E^{1 / p} (r_i)
	}{
		r_i^{(n + \varkappa - m) / p}
	},
$$
whence in accordance with~\eqref{connection} and~\eqref{PT2.2.9new} we again arrive at~\eqref{PT2.2.10new}.

By Lemma~\ref{L3.5} and estimate~\eqref{PL3.6.1}, the left-hand side of~\eqref{PT2.2.10new} tends to zero as $i \to \infty$. 
Thus, passing in~\eqref{PT2.2.10new} to the limit as $i \to \infty$, one can conclude that
\begin{equation}
	\lim_{i \to \infty}
	\frac{
		E (r_i)
	}{
		r_i^{n - (k - 1) \varkappa}
	}
	=
	0.
	\label{new2}
\end{equation}

We take
$$
	\zeta_i = E (r_i), 
	\quad
	i = 0,1,2,\ldots,
$$
this time.
In so doing, let
$
	\lambda = n - (k - 1) \varkappa
$
as before.
Then inequality~\eqref{PT2.2.6new} can obviously be written in the form
\begin{equation}
	\frac{
		\zeta_{i+1} - \zeta_i
	}{
		\zeta_i^\gamma
	}
	\ge
	C
	h 
	\left(
		\frac{
			\sigma \zeta_i
		}{
			r_i^\lambda
		}
	\right),
	\quad
	i = 0,1,2,\ldots.
	\label{PT2.2.11new}
\end{equation}

Thus, repeating the proof of estimate~\eqref{newPT2.1.6} with~\eqref{PT2.2.12new} replaced by~\eqref{PT2.2.11new}, we obtain
$$
	\int_{
		\zeta_{s_1}
	}^\infty
	\frac{
		d\zeta
	}{
		\zeta^{
			\gamma
		}
	}
	\ge
	C
	\int_0^{
		\zeta_{s_1} / r_{s_1}^\lambda
	}
	\frac{
		\tilde h (\sigma \zeta)
	}{
		\zeta
	}
	d\zeta.
$$
The left-hand side of the last expression is less than infinity as
$$
	\gamma
	=
	\frac{
		n + \varkappa
	}{
		n - \varkappa (k - 1)
	}
	>
	1
$$
according to the first equation in~\eqref{PT2.2.1new}, whereas
\begin{equation}
	\int_0^{
		\zeta_{s_1} / r_{s_1}^\lambda
	}
	\frac{
		\tilde h (\sigma \zeta)
	}{
		\zeta
	}
	d\zeta
	\ge
	\frac{
		1
	}{
		2^{\lambda \gamma}
	}
	\int_0^{
		\zeta_{s_1} / r_{s_1}^\lambda
	}
	\frac{
		f (\sigma \zeta)
	}{
		\zeta^{
			1 + \gamma
		}
	}
	d\zeta
	=
	\infty
	\label{PT2.2.14new}
\end{equation}
in view of~\eqref{PT2.2.13new} and condition~\eqref{T2.1.1}.
This contradiction proves the theorem.
\end{proof}

\begin{proof}[Proof of Theorem~\ref{T2.2}]
Assume the converse. 
Let $u$ be a non-trivial global weak solution of~\eqref{1.1i}, \eqref{1.1c}
and let $q$ be some real number such that $p \le q < p / (1 - m / n)$. We put
\begin{equation}
	\varkappa
	=
	\frac{
		(p / q - 1) n + m
	}{
		1 + (k - 1) p / q
	}
	\label{varkappanew}
\end{equation}
and
\begin{equation}
	\gamma
	=
	\frac{
		n p + m q / k
	}{
		n - m (1 - 1 / k)
	}.
	\label{newgamma}
\end{equation}
It is easy to see that $\varkappa > 0$ and $\gamma > 0$.
In so doing, $\varkappa$ and $\gamma$ satisfy the system of equations
\begin{equation}
	\left\{
		\begin{aligned}
			&
			n + \varkappa - \frac{\gamma}{q} (n - (k - 1) \varkappa) = 0,
			\\
			&
			n + \varkappa - \frac{\gamma}{p} (n + \varkappa - m) = 0.
		\end{aligned}
		\label{PT2.2.1}
	\right.
\end{equation}
In particular, this system implies that
\begin{equation}
	\frac{
		n + \varkappa - m
	}{
		p
	}
	=
	\frac{
		n - (k - 1) \varkappa
	}{
		q
	}.	
	\label{PT2.2.2}
\end{equation}

Take a real number $r_0 > 0$ such that $E (r_0) > 0$, where the function $E$ is defined by~\eqref{E}. 
Also let $r_i = 2^i r_0$, $i = 1,2,\ldots$.
As in the proof of Theorem~\ref{T2.1}, we assume that the constants $C$ and $\sigma$ below can also depend on $E (r_0)$.

Let us show that relation~\eqref{new2} remains valid. Indeed, by Lemma~\ref{L3.9}, for any integer $i \ge 0$ at least one of the two systems of inequalities~\eqref{PT2.2.7new} or~\eqref{PT2.2.8new} is satisfied.
In the case where~\eqref{PT2.2.7new} holds, the first inequality in~\eqref{PT2.2.7new} implies~\eqref{PT2.2.10new}. In its turn, if~\eqref{PT2.2.8new} holds, then the first inequality in~\eqref{PT2.2.8new} yields
$$
	\left(
		\frac{
			1
		}{
			\operatorname{mes} Q_{r_{i+1}}^\varkappa \setminus Q_{r_i}^\varkappa
		}
		\int_{
			Q_{r_{i+1}}^\varkappa \setminus Q_{r_i}^\varkappa
		}
		|u|
		dx
		dt
	\right)^q
	\ge
	\frac{
		\sigma
		E^{q / p} (r_i)
	}{
		r_i^{q (n + \varkappa - m) / p}
	},
$$
whence in accordance with~\eqref{PT2.2.2} and the fact that $E^{q / p} (r_i) \ge E^{q / p - 1} (r_0) E (r_i)$, we have
\begin{equation}
	\left(
		\frac{
			1
		}{
			\operatorname{mes} Q_{r_{i+1}}^\varkappa \setminus Q_{r_i}^\varkappa
		}
		\int_{
			Q_{r_{i+1}}^\varkappa \setminus Q_{r_i}^\varkappa
		}
		|u|
		dx
		dt
	\right)^q
	\ge
	\frac{
		\sigma
		E (r_i)
	}{
		r_i^{n - (k - 1) \varkappa}
	}.
	\label{PT2.2.4}
\end{equation}
Hence, to establish the validity of~\eqref{new2}, it sufficient to note that the left-hand sides of~\eqref{PT2.2.10new} and~\eqref{PT2.2.4} tend to zero as $i \to \infty$ by Lemma~\ref{L3.5} and relation~\eqref{PL3.6.1}.

Now, we show the validity of the estimate
\begin{equation}
	\frac{
		E (r_{i+1}) - E (r_i)
	}{
		E^{\gamma / q} (r_i)
	}
	\ge
	C
	h 
	\left(
		\sigma 
		\left(
			\frac{
				E (r_i)
			}{
				r_i^{n - (k - 1) \varkappa}
			}
		\right)^{1 / q}
	\right)
	\label{PT2.2.5}
\end{equation}
for all $i = 0,1,2,\ldots$, where $h$ is defined by~\eqref{h}. 
Indeed, in view of~\eqref{new2}, there exits an integer $i_0 \ge 0$ such that 
\begin{equation}
	\frac{
		E (r_i)
	}{
		r_i^{n - (k - 1) \varkappa}
	}
	\le 
	1
	\label{PT2.2.6}
\end{equation}
for all $i \ge i_0$. Without loss of generality, it can be assumed that $i_0 = 0$; otherwise we replace $r_0$ by $r_{i_0}$.

If~\eqref{PT2.2.7new} holds, then in accordance with~\eqref{PT2.2.6} and the second inequality in~\eqref{PT2.2.7new} we obtain
\begin{equation}
	E (r_{i+1}) - E (r_i)
	\ge
	C
	r_i^{n + \varkappa}
	f 
	\left(
		\sigma 
		\left(
			\frac{
				E (r_i)
			}{
				r_i^{n - (k - 1) \varkappa}
			}
		\right)^{1 / q}
	\right).
	\label{PT2.2.8}
\end{equation}
By the first equation in~\eqref{PT2.2.1}, this immediately implies~\eqref{PT2.2.5}.
In its turn, if~\eqref{PT2.2.8new} is satisfied, then~\eqref{PT2.2.2} and the second inequality in~\eqref{PT2.2.8new} yield
$$
	E (r_{i+1}) - E (r_i)
	\ge
	C
	r_i^{n + \varkappa}
	f 
	\left(
		\frac{
			\sigma 
			E^{1 / p} (r_i)
		}{
			r_i^{(n - (k - 1) \varkappa) / q}
		}
	\right).
$$
Since $E^{1 / p} (r_i) \ge E^{1 / p - 1 / q} (r_0) E^{1 / q} (r_i)$, this again leads us to~\eqref{PT2.2.8}, whence~\eqref{PT2.2.5} follows at once. 

Let us denote
$
	\lambda = (n - (k - 1) \varkappa) / q
$
and
$$
	\zeta_i = E^{1 / q} (r_i), 
	\quad
	i = 0,1,2,\ldots.
$$
Taking into account~\eqref{varkappanew} and the condition $n > m (1 - 1 / k)$, we have
$$
	\lambda
	=
	\frac{
		k (n - m (1 - 1 / k))
	}{
		q + (k - 1) p
	}
	>
	0.
$$
In the above notation,~\eqref{PT2.2.5} can obviously be written as
\begin{equation}
	\frac{
		\zeta_{i+1}^q - \zeta_i^q
	}{
		\zeta_i^\gamma
	}
	\ge
	C
	h 
	\left(
		\frac{
			\sigma 
			\zeta_i
		}{
			r_i^\lambda
		}
	\right),
	\quad
	i = 0,1,2,\ldots.
	\label{PT2.2.7}
\end{equation}
Thus, repeating the proof of estimate~\eqref{newPT2.1.6} with~\eqref{PT2.2.12new} replaced by~\eqref{PT2.2.7}, we obtain
$$
	\int_{
		\zeta_{s_1}
	}^\infty
	\frac{
		d\zeta
	}{
		\zeta^{
			{\gamma - q + 1}
		}
	}
	\ge
	C
	\int_0^{
		\zeta_{s_1} / r_{s_1}^\lambda
	}
	\frac{
		\tilde h (\sigma \zeta)
	}{
		\zeta
	}
	d\zeta.
$$
From~\eqref{newgamma}, it follows that $\gamma \to p / (1 - m / n)$ as $q \to p / (1 - m / n)$. Therefore, the real number $q \in [p, p / (1 - m / n))$ can be chosen such that $\gamma > \mu$. In view of~\eqref{T2.2.1} and~\eqref{PT2.2.13new}, this implies~\eqref{PT2.2.14new}.
On the other hand, from the first equation in~\eqref{PT2.2.1}, it follows that $\gamma > q$; therefore,
$$
	\int_{
		\zeta_{s_1}
	}^\infty
	\frac{
		d\zeta
	}{
		\zeta^{
			{\gamma - q + 1}
		}
	}
	<
	\infty.
$$
This contradiction completes the proof.
\end{proof}

\begin{proof}[Proof of Theorem~\ref{T2.3}]
It is obvious that~\eqref{1.1i}, \eqref{1.1c} can have only trivial solutions. 
This follows from Theorem~\ref{T2.0} if $n \le m (1 - 1 / k)$ or from Theorems~\ref{T2.1} and~\ref{T2.2} if $n > m (1 - 1 / k)$.
At the same time, problem~\eqref{1.1i}, \eqref{1.1c} can not have trivial solutions. Really, if $f (0) > 0$, then for any non-negative function $\varphi \in C_0^\infty ({\mathbb R}_+^{n+1})$ with a non-empty support the right-hand side of~\eqref{1.2} is positive, whereas the left-hand side is equal to zero.
\end{proof}

\begin{remark}\label{R3.1}
Theorems~\ref{T2.0}--\ref{T2.3} remain valid if, instead of the non-negativity of the initial value $u_{k-1}$, one requires that 
$$
	U (r)
	=
	\int_{B_r}
	u_{k-1}
	dx
$$
is a non-decreasing non-negative function in a neighborhood of infinity.
In this case, we have
$$
	U' (r)
	=
	\int_{S_r}
	u_{k-1}
	d S_r
	\ge
	0
$$
for almost all $r$ in a neighborhood of infinity, where $S_r$ is the sphere of radius $r > 0$ centered at zero and $d S_r$ is the volume element on $S_r$.
Therefore, Lemma~\ref{L3.1} and all subsequent statements remain valid if the real number $R$ in their conditions is large enough. Indeed, taking the test function $\varphi$ such as in the proof of Lemma~\ref{L3.1}, we obtain that the first summand on the right in~\eqref{1.2} is non-negative. 
This is sufficient for the validity of inequality~\eqref{L3.1.1}.
\end{remark}

\bigskip
\bigskip

{\bf Acknowledgments} The work of the first author was supported by the Russian Ministry of Education and Science as part of the program of the Moscow Center for Fundamental and Applied Mathematics under the agreement 075-15-2022-284  (critical exponents),
and by the Russian Science Foundation, project 20-11-20272-$\Pi$ (estimates of global solutions).
The work of the second author was supported by the Russian Science Foundation, project 23-11-00056 (asymptotic properties of solutions).

\bigskip

\noindent
{\bf Data availability} Data sharing not applicable to this article as no datasets were generated or analyzed during the current study.

\bigskip

\noindent
{\large \bf Declarations}

\medskip

\noindent
{\bf Conflict of interest} On behalf of all authors, the corresponding author states that there is no conflict of interest.

\end{document}